\documentclass[10pt]{amsart}
\usepackage{float}
\usepackage{amssymb}
\usepackage{graphicx}
\usepackage{caption}

\usepackage[usenames,dvipsnames]{color}
\usepackage{soul}
\usepackage{hyperref}
\usepackage{tikz}

\def\N{{\mathbb{N}}}
\def\Z{{\mathbb{Z}}}

\newtheorem{example}{Example}
\numberwithin{example}{subsection}

\numberwithin{eg}{subsubsection}

\newtheorem{corollary}{Corollary}

\newtheorem{theorem}{Theorem}
\newtheorem{lemma}{Lemma}

\title{Pregroups and hyperbolicity}

\author{Jiayue Li}
\address{Department of Mathematical Sciences, Stevens Institute of Technology, 1 Castle Point on Hudson, Hoboken, NJ 07030, USA} 
\email{jli211@stevens.edu}

\author{Denis Serbin}
\address{Department of Mathematical Sciences, Stevens Institute of Technology, 1 Castle Point on Hudson, Hoboken, NJ 07030, USA} 
\email{d.e.serbin@gmail.com}

\keywords{Pregroup, hyperbolic group, combination theorem}

\subjclass[2010]{20F65, 20F67, 20E06}

\date{}

\begin{document}

\begin{abstract}
In this paper we study hyperbolicty of the universal group $U(P)$ of a pregroup $P$. Given a finitely generated group $G$ and a pregroup $P$ such that $G \simeq U(P)$, we provide a particular set of axioms on $P$ which ensure that $G$ is word-hyperbolic.
\end{abstract}

\maketitle

\tableofcontents

\section{Introduction}
\label{sec:intro}

The idea of representing group elements by certain normal forms dates back the papers of B. L. van der Waerden \cite{VanDerWaerden:1948} and R. Baer \cite{Baer:1949, Baer:1950a,Baer:1950b}, where the authors studied properties of normal forms of elements in free products and free products with amalgamation. Then R. Lyndon, in his paper \cite{Lyndon:1963}, introduced groups equipped with abstract length functions as a 
tool to carry over Nielsen cancellation theory from free groups and free products, to a much more general setting. Eventually, another attempt to find axioms formalizing normal forms on groups was made by J. Stallings in the paper \cite{Stallings:1971}, where he tried to generalize the notion of amalgamated free product. He introduced {\em pregroup} as a set with partially defined multiplication, which is associative whenever possible. Such an algebraic structure has a multiplicative identity together with the multiplicative inverse for every element, so that certain axioms are satisfied. Now, every pregroup $P$ is associated with its {\em universal group} $U(P)$, which is in some sense the largest group generated by $P$. Elements of $U(P)$ can be represented as finite reduced products over the alphabet $P$, which is a generalization of reduced normal forms for elements of free product with amalgamation.

\smallskip

Since every representation of $g \in U(P)$ as a finite reduced product $p_1 \cdots p_k$ over $P$ has the same number of factors, the function $l: U(P) \to \Z$ given by $l(g) = k$ is a well-defined length function on $U(P)$. It was studied by I. Chiswell in \cite{Chiswell:1987} and A. H. M. Hoare in \cite{Hoare:1988} (see also \cite{Hurley:1978, Nesayef:1983}). In particular, it was proved that the function $l$ is a {\em Lyndon length function} if the underlying pregroup $P$ satisfies an extra axiom (see \cite{Chiswell:1987, Hoare:1988} for details). It is known that existence of a Lyndon length function on a group is equivalent to an action of the group on a simplicial tree (see \cite{Chiswell:1976}), which can be studied using Bass-Serre theory (see \cite{Serre:1980}). Connections between pregroups and Bass-Serre theory were thoroughly investigated by F. Rimlinger in \cite{Rimlinger:1987a}, where the author came up with another condition on a pregroup $P$ under which the universal group $U(P)$ has a non-trivial action on a simplicial tree. At about the same time attempts to generalize Stallings' pregroups were made by H. Kushner and S. Lipschutz in \cite{Kushner_Lipschutz:1988}, where they defined {\em pree} as a set with a partial operation satisfying axioms that are weaker than pregroup axioms, but still strong enough to prove that a pree embeds into its universal group.

\smallskip

Connections between pregroups and Lyndon length functions on their universal groups were exploited in \cite{Myasnikov_Remeslennikov_Serbin:2005, KMRS:2012, KMS:2016}, where the authors studied groups of infinite non-Archimedean words. Another application of pregroup techniques was demonstrated in \cite{Kvaschuk_Myasnikov_Serbin:2009}, where groups with so-called {\em big powers condition} were studied. 

\smallskip

Pregroups provide a convenient algebraic tool, which makes it possible to conduct various kinds of combinatorial arguments based on cancellation between normal forms representing elements of the universal group. As was shown in \cite{Chiswell:1987} and \cite{Kvaschuk_Myasnikov_Serbin:2009}, properties of pregroups defined by certain extra axioms translate into properties of their universal groups: existence of a splitting and the big power property of $U(P)$ can be encoded by axioms on $P$. It is interesting to find other group properties of $U(P)$ that can be defined by axioms on $P$. It is also interesting to connect pregroup techniques with methods of geometric group theory: extra axioms on $P$ may have impact on geometric propeties of the Cayley graph of $U(P)$, and it is natural to ask what geometric properties of $U(P)$ can be encoded by axioms on $P$. 

In this paper we study connections between properties of pregroup structures on a finitely generated group $G$ and hyperbolicity of $G$. We show that if $P$ satisfies certain axioms (see Section \ref{subsec:hyp-universal} for details), then $G \simeq U(P)$ is word-hyperbolic. More precisely, we prove the following result

{\bf Theorem.}\ {\em Le $G$ be a group with a finite generating set $S$ and suppose $G$ has a pregroup structure with respect to a pregroup $P$. If $P$ satisfies the axioms (P6), (H0) - (H3), then there exists $\delta > 0$ such that every geodesic triangle in $\Gamma(G, S)$ is $\delta$-thin. In particular, $G \simeq U(P)$ is word-hyperbolic.}

\section{Preliminaries}
\label{sec:preliminaries}

In this section we recall the definition of a {\em pregroup} and its {\em universal group}, and formulate the main goal of the paper.

\subsection{Pregroups and their universal groups}
\label{subsec:pregroups}

In \cite{Stallings:1971} J. Stallings introduced a notion of a pregroup $P$ and its universal group $U(P)$. A pregroup $P$ provides a very economic way to describe reduced forms of elements of $U(P)$.

A {\em pregroup} $P$ is a set $P$, with a distinguished element $1$, equipped with  a partial multiplication, that is a function $D \rightarrow P$, $(x,y) \rightarrow x y$, where $D \subset P \times P$, and an inversion, that is, a function $P \rightarrow P$, $x \rightarrow x^{-1}$, satisfying the following axioms (below $x y$ is {\em defined} if $(x, y) \in D$):

\begin{enumerate}
\item[(P1)] For all $u \in P,~$ $u 1$ and $1 u$ are defined and $u 1 = 1 u = u$.

\item[(P2)] For all $u \in P,~$ $u^{-1} u$ and $u u^{-1}$ are defined and $\ \ u^{-1} u = u u^{-1} = 1$.

\item[(P3)] For all $u, v \in P$, if $u v$ is defined, then so is $v^{-1} u^{-1}$, and $(u v)^{-1} = v^{-1} u^{-1}$.

\item[(P4)] For all $u, v, w \in P$, if $u v$ and $v w$ are defined, then $(u v) w$ is defined if and only if $u (v w)$ is defined, in which case
$$(u v) w =  u (v w).$$

\item[(P5)] For all $u,v,w,z \in P$, if $u v, v w$, and $w z$ are all defined then either $u v w$, or $v w z$ is defined.
\end{enumerate}

In fact (see \cite{Hoare:1988}), (P3) follows from (P1), (P2), and (P4), hence, it can be omitted.

To describe the universal group $U(P)$ recall that a mapping $\phi: P \rightarrow Q$ of pregroups is a {\em morphism} if for any $x, y \in P$ whenever $x y$ is defined in $P$, the product $\phi(x) \phi(y)$ is defined in $Q$ and it is equal to $\phi(x y)$.

Now the group $U(P)$ can be characterized by the following universal property: there is a morphism of pregroups $\lambda: P \rightarrow U(P)$, such that, for any morphism  $\phi: P \rightarrow G$ of $P$ into a group $G$, there is a unique group homomorphism $\psi: U(P) \rightarrow G$ for which $\psi \lambda  = \phi$. This shows that $U(P)$ is a group with a generating set $P$ and a set of relations $x y = z$, where $x, y \in P$, $x y$ is defined in $P$, and equals $z$.

There exists an explicit construction of $U(P)$ due to Stallings \cite{Stallings:1971}. A finite sequence $u_1, \ldots, u_n$ of elements from $P$ is termed a {\em $P$-product} and it is denoted by $u_1 \cdots u_n$ (one may view it as a finite word in the alphabet $P$). A $P$-product $u_1 \cdots u_n$ is called {\em reduced} if for every $i \in [1, n-1]$ the product $u_i u_{i+1}$ is not defined in $P$. In the case when $u_i u_{i+1}$ is defined in $P$ and equals $v \in P$, one can {\em reduce} $u_1 \cdots u_n$ by replacing the pair $u_i u_{i+1}$ by $v$.  Let ``$\sim$'' be the equivalence relation on the set of all reduced $P$-products defined as follows: $u_1 \cdots u_n \sim v_1 \cdots v_m$ if and only if $m = n$ and there exist elements $a_1, \ldots, a_{n-1} \in P$ such that the product $a_{i-1}^{-1} u_i a_i$  defined and $v_i = a_{i-1}^{-1} u_i a_i$ for every $i \in [1, n]$ (here $a_0 = a_n = 1$). Then the group $U(P)$ can be described as the set $U(P)$ of equivalence classes (with respect to $\sim$) of reduced $P$-products, where multiplication is given by concatenation of representatives and consecutive reduction of the resulting product. Obviously, $P$ embeds into $U(P)$ via the canonical map $u \rightarrow u$.

Below we give some natural examples of pregroups and their universal groups.

\begin{example}
Any group $G$ is a pregroup $P$, where $P = G$ and $D = G \times G$.
\end{example}

\begin{example}
For any set $X$ let $X^{-1}$ be the set of formal inverses. Define $P = X \cup X^{-1} \cup \{1\}$ and $x y$ is defined in $P$ only if either $x = 1$, or $y = 1$, or $x = y^{-1}$. Hence, $P$ is a pregroup and $U(P) \simeq F(X)$.
\end{example}

\begin{example}
For any groups $A, B$ such that $A \cap B = C$, define $P = A \cup B$ and $x y$ is defined in $P$ only if either $x, y \in A$, or $x,y \in B$. Hence, $P$ is a pregroup and $U(P) \simeq A \ast_C B$.
\end{example}

\begin{example}
Any HNN-extension $G = \langle H, t \mid t^{-1} A t = \phi(A) \rangle$ is isomorphic to $U(P)$, where 
$$P = \{ h_1 t^\varepsilon h_2 \mid h_1, h_2 \in H,\ \varepsilon \in \{-1,0,1\} \}$$
and $D$ consists of pairs $(x,y) \in P \times P$ such that the product $x y$ can be reduced in $G$ back to an element from $P$.
\end{example}

\subsection{Properties of the universal group defined by pregroup properties}
\label{subsec:extra-axioms}

In this section we recall what properties of the universal group $U(P)$ of a pregroup $P$ are defined by properties of $P$. 

\subsubsection{Splittings of the universal group}
\label{subsubsec:split}

Given a group $G$, we say that {\em $G$ splits (into an amalgamated free product or an HNN extension)} if either $G \simeq A \ast_C B$ for some groups $A,\ B$ and $C \neq 1$, or $G \simeq H \ast_C$ for some $H$ and $C \neq 1$.

Let $P$ be a pregroup. It turns out (see \cite{Rimlinger:1987a, Chiswell:1987}) that existence of a splitting of the universal group $U(P)$ of $P$ depends on certain properties of $P$.

Following \cite{Stallings:1971}, for every $x \in P$ define 
$$L(x) = \{a \in P \mid (a, x) \in D\}.$$ 
Next, we write $x \leqslant y$ if and only if $L(y) \subseteq L(x)$ and define a relation on $P$: $x \sim y$ if $L(x) = L(y)$. If $x \leqslant y$ and $x \not\sim y$, then we write $x < y$.

``$\sim$'' is an equivalence relation on $P$ and $P/ \sim$ is a partially ordered set (the order is induced from $P$). $P/ \sim$ is called the {\em order tree of $P$}. $P/ \sim$ is of {\em finite height} if for every $x \in P$ there are only finitely many $x_0, \ldots, x_k \in P$ such that $1 = x_0 < \cdots < x_{n-1} < x_k = x$. 

$P/ \sim$ of finite height can be turned into a simplicial tree $T_P$, where each equivalence class $[x]$ represents a vertex $v_x$. A pair $(v_x, v_y)$ defines a positively oriented edge if $x < y$ and for any $x \leqslant z \leqslant y$, it follows that either $z \sim x$, or $z \sim y$.

\begin{theorem}\cite{Rimlinger:1987a, Hoare:1988}
\label{th:split-1}
If $P/ \sim$ is of finite height, then it is possible to define a non-trivial action of $U(P)$ on some $\Z$-tree, that is, $U(P)$ can be represented as the fundamental group of a graph of groups.
\end{theorem}

The Lyndon length function associated with the action above is defined as follows. For every $w = x_1 \cdots x_n \in U(P)$ we have
$$l(x_1 \cdots x_n) = \sum_{i=1}^{n+1} d(x^{-1}_{i-1}, x_i),$$
where $x_0 = x_{n+1} = 1$ and $d(x,y)$ is the distance between $[x]$ and $[y]$ in $T_P$.

According to the Theorem \ref{th:split-1}, the universal group of a pregroup of finite height has a splitting. It turns out (see \cite{Chiswell:1987}) that existence of a splitting of $U(P)$ can be encoded by an axiom on $P$.

Following \cite{Chiswell:1987}, for a pregroup $P$ define
$$B_P = \{b \in P \mid z b\ {\rm and}\ b z\ {\rm are\ defined\ for\ all}\ z \in P\}.$$
Obviously, $B_P$ is a subgroup of $P$. Furthermore, if a reduced $P$-product contains an element from $B_P$, then it consists of a single element.

\begin{example}
Let $G = G_1 \ast_C G_2$. Then $G \simeq U(P)$, where $P = G_1 \cup G_2$ and $(x, y) \in D$ only if either $x, y \in G_1$ or $x, y \in G_2$. In this case $B_P = C$.
\end{example}

Note that if a $P$-product $u = u_1 \cdots u_n$ is reduced, then the function $l: U(P) \to \Z$ that sends $u$ to $n$ is a well-defined function on $U(P)$. We write $|u| = n$ if $l(u) = n$.

It was shown in \cite{Chiswell:1987} that the function $l: U(P) \to \Z$ is a Lyndon length function if and only if $P$ satisfies an additional axiom (P6):

\begin{enumerate} \label{pr:P6}
\item[(P6)] For any $x,y \in P$, if $x y$ is not defined, but $x a$ and $a^{-1} y$ are both defined for some $a \in P$, then $a \in B_P$.
\end{enumerate}

Namely, the following result is known.

\begin{theorem}\cite{Chiswell:1987}
\label{th:split-2}
Let $P$ be a pregroup and $B_P$ defined as above. Then the function $l : U(P) \to \Z$ defined below for any reduced $P$-product $u_1 u_2 \cdots u_n$
\[l(u_1 u_2 \cdots u_n) = \left\{\begin{array}{ll}
\mbox{$n$,} & \mbox{if $n > 1$} \\
\mbox{$1$,} & \mbox{if $n = 1$ and $u_1 \notin B_P$} \\
\mbox{$0$,} & \mbox{if $n = 1$ and $u_1 \in B_P$}
\end{array}
\right.
\]
is a $\Z$-valued Lyndon length function if and only if the pregroup $P$ satisfies the axiom (P6).
\end{theorem}

It is also known (see \cite[Theorem 2.7]{Nesayef:1983}) that the axiom (P6) is equivalent to the following one:

\begin{enumerate} \label{pr:P6'}
\item[(P6')] For any $x, y \in P$, if $x y$ is not defined and $(a x) y$ is defined for some $a \in P$, then $a x \in B_P$.
\end{enumerate}

It is easy to see that if $P$ satisfies (P6), then in a reduced $P$-product $v_1 \cdots v_n$ obtained from another reduced $P$-product $u_1 \cdots u_n$ by interleaving $v_i = a_{i-1}^{-1} u_i a_i$ for $i \in [1,n]$, where $a_0 = a_n = 1$, the elements $a_1, \ldots, a_{n-1}$ belong to $B_P$.

\subsubsection{The big powers condition}
\label{subsubsec:bp}

Let $G$ be a group. The {\em big powers (or BP for short) condition} holds in $G$ if for every sequence of elements $u_1, \ldots, u_k \in G$ such that $[u_i, u_{i+1}] \neq 1$ for $i \in [1,k-1]$, there exists $n \in \N$ such that $u_1^{\alpha_1} \cdots u_k^{\alpha_k} \neq 1$ for any $\alpha_1, \ldots, \alpha_k \geqslant n$. In this case we call $G$ a {\em BP-group}.

A particular form of the BP condition was explicitly introduced by G. Baumslag in \cite{Baumslag:1962} and then A. Ol'shanskii in \cite{Ol'shanskii:1993} defined the most general form of the condition.

Below we give some examples of classes of BP-groups.

\begin{enumerate}
\item Torsion-free abelian groups and torsion-free hyperbolic groups have BP.
\item Subgroups of BP-groups are also BP-groups.
\item Groups discriminated by BP-groups are BP-groups.
\end{enumerate}

Note also that the BP-property is not preserved by direct products.

In \cite{Kvaschuk_Myasnikov_Serbin:2009} the following question was asked: What axioms should be imposed on the pregroup $P$ to guarantee the BP property in $U(P)$? In order to formulate the axioms sufficient for the BP condition, let us recall the {\em isolation} and {\em strong isolation} subgroup properties. A subgroup $A$ of a group $G$ is called {\em isolated} if for any $g \in G,\ k \in \Z, \ k \neq 0$: $\ g^k \in A$ implies $g \in A$. Similarly, a subgroup $A$ of a group $G$ is called {\em strongly isolated} if for every sequence $u_1, \ldots, u_k \in G$, in which $[u_i, u_{i+1}] \neq 1,\ i \in [1, k-1]$ and not all $u_i$ belong to $A$, there exists $n \in \N$ such that $u_1^{\alpha_1} \cdots u_k^{\alpha_k} \notin A$ for any $\alpha_1, \ldots, \alpha_k \geqslant n$. 
Here is a list of some basic facts related to these properties.

\begin{enumerate}
\item If $A$ is strongly isolated in $G$, then $G$ is a BP-group if and only if $A$ is.
\item Any centralizer of a BP-group is strongly isolated.
\item A free factor of a BP-group is strongly isolated.
\item Let $G$ be a torsion-free hyperbolic group and $A$ a quasi-convex isolated subgroup of $G$. Then $A$ is strongly isolated.
\end{enumerate}

Now, given a pregroup $P$, we define the following axioms:

\begin{enumerate}
\item[(A1)] {(\em isolation of $B_P$)} For any $p \in P$, if $p^n$ is defined for some $n \in \Z$ and $p^n \in B_P$, then $p \in B_P$.

\item[(A2)] {\em (strong isolation of $B_P$)} For any sequence $p_1, \ldots, p_n \in P$, where $[p_i, p_{i+1}] \neq 1,\ i \in [1,k-1]$ and not all $p_i$ belong to $B_P$, there exists $m \in \N$ such that $p_1^{\alpha_1} \cdots p_n^{\alpha_n} \notin B_P$ for all $\alpha_i \geqslant m$.

\item[(A3)] {\em ($B_P$ is torsion-free)} For any $b \in B_P,\ n \in \N$, if $b^n = 1$, then either $b = 1$, or $n = 0$.

\item[(A4)] {\em (malnormality of $B_P$)} Either $B_P = 1$, or for any $p \in P - B_P,\ a \in B_P$, the product $p^{-1} a p\ $ either is not defined, or it is defined and does not belong to $B_P$.

\item[(A5)] For any $p \in P,\ a \in B_P$, the product $p^{-1} a p$ is defined.
\end{enumerate}

The axioms listed above are enough to guarantee the big powers condition on the universal group $U(P)$ as shown in the theorem below.

\begin{theorem}\cite{Kvaschuk_Myasnikov_Serbin:2009}
Let $P$ be a pregroup which satisfies the axioms (P6), (A1) - (A5) and let $U(P)$ be a CSA-group. Then $U(P)$ is a BP-group if and only if $B_P$ is.
\end{theorem}

The theorem can be applied to amalgamated free products, where the amalgamated subgroup is strongly isolated and malnormal in both factors.

\begin{corollary}\cite{Kvaschuk_Myasnikov_Serbin:2009}
Let $G = G_1 \ast_C G_2$, where $G_1, G_2$ are torsion-free CSA-groups, $C$ is strongly isolated and malnormal in both $G_1$ and $G_2$. Then $G$ is a BP-group if and only if $C$ is.
\end{corollary}

\subsection{Hyperbolicity of the universal group}
\label{subsec:hyp-universal}

Recall that a finitely generated group $G$ is called {\em (word) hyperbolic} if its Cayley graph $\Gamma(G, S)$ with respect to some finite generating set $S$ is a {\em $\delta$-hyperbolic space}, that is, it satisfies the following property: there exists a constant $\delta > 0$, called a {\em constant of hyperbolicity} of $G$, such that every geodesic triangle $\Delta(x, y, z)$ with vertices $x, y, z \in \Gamma(G, S)$ is $\delta$-thin meaning that each side of $\Delta(x, y, z)$ lies inside the union of $\delta$-neighborhoods of the other two sides. If $S$ is fixed and such constant $\delta$ exists then we also call $G$ {\em $\delta$-hyperbolic}.

$\delta$-hyperbolicity of $\Gamma(G, S)$ (and $G$ itself) can also be formulated in terms of the {\em Gromov product}. Recall that given $x, y, v \in \Gamma(G, S)$, the {\em Gromov product} of $x$ and $y$ with respect to $v$ is defined as
$$(x \cdot y)_v = \frac{1}{2}(d(x, v) + d(y, v) - d(x, y)).$$
Now $\Gamma(G, S)$ is a $\delta$-hyperbolic space (that is, $G$ is a hyperbolic group) if and only if for all $x, y, z, v \in \Gamma(G, S)$
$$(x \cdot y)_v \geqslant \min\{(x \cdot z)_v, (z \cdot y)_v\} - 2\delta.$$

Hyperbolic groups were introduced by Gromov in \cite{Gromov:1987} and now it is a well-studied class of groups (we refer the reader to \cite{Gromov:1987}, \cite{ABCFLMSS:1990}, \cite{Ghys_delaHarpe:1991}, \cite{Bridson_Haefliger:1999} for basic facts about hyperbolic groups).

In this section we introduce several axioms on a pregroup and formulate the main result of the paper: if a pregroup structure on a finitely generated group $G$ satisfied the introduced axioms, then $G$ is a hyperbolic group.

\bigskip

Let $G$ be a group with a finite generating set $S$ and let $G \simeq U(P)$ for some pregroup $P$. In this case we say that $G$ has a {\em pregroup structure with respect to $P$}. The Cayley graph $\Gamma(G, S)$ with respect to $S$ is a metric space with the word metric $d$. If $g \in G$, then we set $|g| = d(1, g)$ in $\Gamma(G, S)$.

We can view $P$ as a subset of $G$ and $(x, y) \in D \subseteq P \times P$ if and only if $x y \in P$.

Note that the structure of $G$ is known in the case when $P$ is finite.

\begin{theorem}\cite{Rimlinger:1987a}
\label{th:fin_pregorpoup}
If $P$ is a finite pregroup, then $U(P)$ is isomorphic to the fundamental group of a finite graph of finite groups.
\end{theorem}

In other words, if $G$ has a pregroup structure with a finite underlying pregroup $P$, then combining Theorem \ref{th:fin_pregorpoup} and the result of Stallings (see \cite{Stallings:1970}), we obtain that $G$ is virtually free. In particular, $G$ is hyperbolic. 

\smallskip

From now on we assume that $P$ is infinite.

Recall that
$$B_P = \{b \in P \mid z b\ {\rm and}\ b z\ {\rm are\ defined\ for\ all}\ z \in P\}.$$
As we already noted above, if the axiom (P6) holds, then in a reduced $P$-product $v_1 \cdots v_n$ obtained from another reduced $P$-product $u_1 \cdots u_n$ by interleaving $v_i = a_{i-1}^{-1} u_i a_i$ for $i \in [1, n]$, where $a_0 = a_n = 1$, the elements $a_1, \ldots, a_{n-1}$ all belong to $B_P$.

The first extra axiom we consider is very natural.

\begin{enumerate} \label{pr:H0}
\item[(H0)] $B_P$ is finite.
\end{enumerate}

Denote 
$$C_0 = \max\{|p| \mid p \in B_p\}.$$

From the axioms (P6) and (H0) combined, it follows that every $g \in G$ has only finitely many representations as reduced $P$-products.

Next, we define hyperbolicty of ``small'' triangles in $\Gamma(G, S)$.

\begin{enumerate} \label{pr:H1}
\item[(H1)] There exists a constant $C_1 > 0$ such that for every $p, q \in P$, any geodesic triangle on the vertices $\{1, p, p q\}$ in $\Gamma(G, S)$ is $C_1$-thin.
\end{enumerate}

In particular, if $q = 1$, then the digon in $\Gamma(G, S)$ formed by any geodesics connecting $1$ and $p \in P$ is $C_1$-thin.

Also, since every geodesic triangle on the vertices $\{f, g, h\}$, where $g = f p,\ h = g q$ and $p, q \in P$ is translate of a geodesic triangle on the vertices $\{1, p, p q\}$, from the axiom (H1) it follows that the triangle $\Delta(f, g, h)$ is $C_1$-thin.

Further, we are going to use the following axiom inspired by \cite[Lemma 21]{Ol'shanskii:1991} about geodesic polygons in hyperbolic spaces.

\begin{enumerate} \label{pr:H2}
\item[(H2)] There exists $C_2 > 0$ such that for any $p, q \in P$, if $(p^{-1} \cdot q)_1 > C_2$, then $p q \in P$.
\end{enumerate}

Finally, we are going to need one more axiom which deals with geodesic paths in $\Gamma(G, S)$ tracing pregroup elements.

\begin{enumerate} \label{pr:H3}
\item[(H3)] Let $p \in P$ and let $\gamma$ be a geodesic path connecting $1$ and $p$ in $\Gamma(G, S)$. Then, every $x \in \gamma$ can be represented as a reduced $P$-product $x = p_1 p_2$, where $|p_2| \leqslant C_3$ and $C_3$ does not depend on $p$ and $\gamma$.
\end{enumerate}

By symmetry, from (H3) it follows that $p^{-1} x$ can also be represented as a reduced $P$-product $p^{-1} x = q_1 q_2$, where $|q_2| \leqslant C_3$. Hence, $p = x (x^{-1} p) = (p_1 p_2) (q_2^{-1} q_1^{-1})$ and if the axiom (P6) holds for $P$, then it follows that $b = p_2 q_2^{-1} \in B_P$ and $p_1 b q_1^{-1} \in P$.

Now we can formulate the main result of the paper.

\begin{theorem} \label{th:main}
Le $G$ be a group with a finite generating set $S$ and suppose $G$ has a pregroup structure with respect to a pregroup $P$. If $P$ satisfies the axioms (P6), (H0) - (H3), then there exists $\delta > 0$ such that every geodesic triangle in $\Gamma(G, S)$ is $\delta$-thin. In particular, $G \simeq U(P)$ is hyperbolic.
\end{theorem}

The proof of the theorem above is going to rely on several lemmas presented in the next section.

\section{Proof of the main result}
\label{sec:proof}

For the rest of the section we assume that $G$ is a group with a finite generating set $S$. Suppose $G$ has a pregroup structure with respect to a pregroup $P$.

\subsection{Auxiliary lemmas}
\label{subsec:lemmas}

In this section we study geometric properties of $\Gamma(G, S)$ that are induced by the axioms (P6) and (H0) - (H3).

\begin{lemma} \label{le:fin-diameter}
Let $P$ satisfy (P6) and (H0). Then there exists $D_1 > 0$ such that for every generator $s \in S$ and any reduced $P$-product $p_1 \cdots p_k$ representing $s$, the diameter of the path in $\Gamma(G, S)$ corresponding to $p_1 \cdots p_k$ is bounded by $D_1$.
\end{lemma}
\begin{proof}
As we already notes before, from the axioms (P6) and (H0) it follows that for every $s \in S$, there are only finitely many reduced $P$-products representing $s$. Every such product has a finite diameter and since $S$ is finite, the statement of the lemma follows.
\end{proof}

From now on given a reduced $P$-product $u = u_1 \cdots u_n$, for every $i \in [1, n]$ we denote $u(i) = u_1 \cdots u_i$. In particular, $u(1) = u_1$ and $u(n) = u$.

\begin{lemma} \label{le:reduced-paths}
Let $P$ satisfy (P6), (H0), and (H1). Then there exists $D_2 > 0$ such that if reduced $P$-products $u_1 \cdots u_n$ and $v_1 \cdots v_n$ define the same element $g \in G$, then the Hausdorff distance between the paths defined by the products $u = u_1 \cdots u_n$ and $v = v_1 \cdots v_n$ in $\Gamma(G, S)$ is bounded by $D_2$.
\end{lemma}
\begin{proof}
Recall that $B_P = \{b \in P \mid z b\ {\rm and}\ b z\ {\rm are\ defined\ for\ all}\ z \in P\}$. Let $g$ be represented by the reduced $P$-products $u = u_1 \cdots u_n$ and $v = v_1 \cdots v_n$. Then there exist $a_1, \ldots, a_{n-1} \in B_P$ such that 
$$u_1= v_1 a_1, \cdots, u_i = a^{-1}_{i-1} v_i a_i, \cdots, u_n = a_{n-1}^{-1} v_n$$
and we have the following picture in $\Gamma(G, S)$.

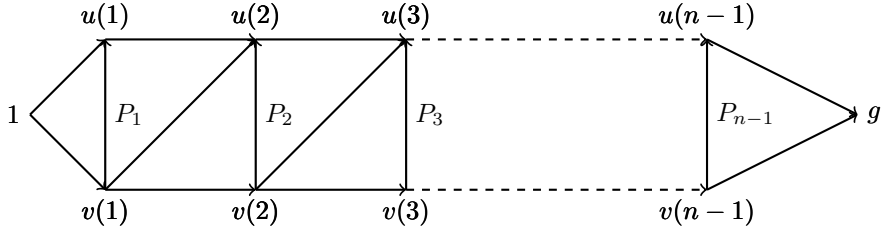
\begin{figure}[H]
\centering
\begin{tikzpicture}
\draw [thick,->] (0,0) node [anchor=east]{$1$} -- (1,1) node [anchor=south]{$u(1)$};
\draw [thick,->] (0,0) node [anchor=east]{$1$} -- (1,-1) node [anchor=north]{$v(1)$};
\draw [thick,->] (1,1) node [anchor=south]{$u(1)$} -- (3,1) node [anchor=south]{$u(2)$};
\draw [thick,->] (1,-1) node [anchor=north]{$v(1)$} -- (3,-1) node [anchor=north]{$v(2)$};
\draw [thick,->] (3,1) node [anchor=south]{$u(2)$} -- (5,1) node [anchor=south]{$u(3)$};
\draw [thick,->] (3,-1) node [anchor=north]{$v(2)$} -- (5,-1) node [anchor=north]{$v(3)$};
\draw [thick,->][thick,dashed] (5,1) node [anchor=south]{$u(3)$} -- (9,1) node [anchor=south]{$u(n-1)$};
\draw [thick,->][thick,dashed] (5,-1) node [anchor=north]{$v(3)$} -- (9,-1) node [anchor=north]{$v(n-1)$};
\draw [thick,->] (9,1) node [anchor=south]{$u(n-1)$} -- (11,0) node [anchor=west]{$g$};
\draw [thick,->] (9,-1) node [anchor=north]{$v(n-1)$} -- (11,0) node [anchor=west]{$g$};
\draw [thick,->] (1,-1) node [anchor=north]{$v(1)$} -- (1,1) node [anchor=south]{$u(1)$};
\draw [thick,->] (3,-1) node [anchor=north]{$v(2)$} -- (3,1) node [anchor=south]{$u(2)$};
\draw [thick,->] (5,-1) node [anchor=north]{$v(3)$} -- (5,1) node [anchor=south]{$u(3)$};
\draw [thick,->] (9,-1) node [anchor=north]{$v(n-1)$} -- (9,1) node [anchor=south]{$u(n-1)$};
\draw [thick,->] (1,-1) node [anchor=north]{$v(1)$} -- (3,1) node [anchor=south]{$u(2)$};
\draw [thick,->] (3,-1) node [anchor=north]{$v(2)$} -- (5,1) node [anchor=south]{$u(3)$};
\draw (1,-1) -- node[midway,right]{$P_1$} (1,1);
\draw (3,-1) -- node[midway,right]{$P_2$} (3,1);
\draw (5,-1) -- node[midway,right]{$P_3$} (5,1);
\draw (9,-1) -- node[midway,right]{$P_{n-1}$} (9,1);
\end{tikzpicture}
\caption{The paths defined by the reduced products $u$ and $v$.}
\label{le2-pic1}
\end{figure}

In the Figure \ref{le2-pic1}, the paths $P_1, \ldots, P_{n-1}$ read the words in $S$ representing the elements $a_1, \ldots, a_{n-1}$ and all the edges are geodesics. Since the label of $P_i$ represents $a_i \in B_P$, the length of each $P_i$ is bounded by a constant $C_0$, by the axiom (H0).

To find a bound on the Hausdorff distance between the paths defined by the products $u$ and $v$ take a point $t$ on the path defined by the product $u$. There exists $i$ such that $t \in [u(i), u(i+1)]$. The cell on the vertices $u(i), u(i+1), v(i), v(i+1)$ is shown below.

\begin{figure}[H]
\centering
\begin{tikzpicture}
\draw [thick,->] (0,0) node [anchor=north]{$v(i)$} -- (4,0) node [anchor=north]{$v(i+1)$};
\draw [thick,->] (0,0) node [anchor=north ]{$v(i)$} -- (0,2) node [anchor=south]{$u(i)$};
\draw [thick,->] (0,2) node [anchor=south ]{$u(i)$} -- (4,2) node [anchor=south]{$u(i+1)$};
\draw [thick,->] (4,0) node [anchor=north]{$v(i+1)$} -- (4,2) node [anchor=south ]{$u(i+1)$};
\draw [thick,->] (0,0) node [anchor=north]{$v(i)$} -- (4,2) node [anchor=south]{$u(i+1)$};
\filldraw[black] (2,0) circle (1.5pt) node[anchor=north]{$s^{\prime}$};
\filldraw[black] (2,2) circle (1.5pt) node [anchor=south]{$t$};
\filldraw[black] (0,1.2) circle (1.5pt) node[anchor=east]{$t^{\prime\prime}$};
\filldraw[black] (4,0.7) circle (1.5pt) node[anchor=west]{$s^{\prime\prime}$};
\filldraw[black] (2,1) circle (1.5pt) node [anchor=north east]{$t^{\prime}$};
\draw [thick,dashed] (0,1.2) -- (2,2);
\draw [thick,dashed] (2,0) .. controls (2.2,0.5) .. (2,1);
\draw [thick,dashed] (2,1) .. controls (1.8,1.5) .. (2,2);
\draw [thick,dashed] (2,1) .. controls (3,0.5) .. (4,0.7);
\end{tikzpicture}
\caption{The cell on the vertices $u(i), u(i+1), v(i), v(i+1)$.}
\label{le2-pic2}
\end{figure}
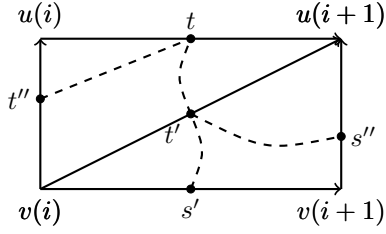

Note that any geodesic triangle $\Delta(u(i) , v(i), u(i+1))$ on the vertices $\{u(i) , v(i), u(i+1) \}$ is $C_1$-thin, which follows from (H1). Hence, either there exists $t' \in [u(i+1), v(i)]$ such that $d(t, t') \leqslant C_1$, or there exists $t'' \in [u(i), v(i)]$ such that $d(t, t') \leqslant C_1$. In the former case, since $d(u(i), v(i)) \leqslant C_0$ by the axiom (H0), we obtain 
$$d(t, v(i)) \leqslant d(t, t') + d(t', v(i)) \leqslant C_0 + C_1.$$ 
In the latter case, consider the geodesic triangle $\Delta(u(i) , v(i), u(i+1))$ on the vertices $\{ v(i) , u(i+1), v(i+1) \}$, which is also $C_1$-thin. Hence, either there exists $s' \in [v(i), v(i+1)]$ such that $d(t', s') \leqslant C_1$, or there exists $s'' \in [u(i+1), v(i+1)]$ such that $d(t', s'') \leqslant C_1$. In the former case we obtain 
$$d(t, s') \leqslant d(t, t') + d(t', s') \leqslant 2C_1,$$ 
and in the latter case we obtain
$$d(t, v(i+1)) \leqslant d(t, t') + d(t', s'') + d(s'', v(i+1)) \leqslant C_1 + C_1 + C_0 = C_0 + 2C_1.$$
Hence, the distance from an arbitrary $t$ on the path defined by the product $u$ to the path defined by the product $v$ is bounded by $C_0 + 2C_1$.

Exactly the same argument works in the opposite direction and we can take $D_1 = C_0 + 2C_1$.
\end{proof}

\begin{lemma} \label{le:thin-tiangles}
Let $P$ satisfy (P6), (H0), and (H1). There exists $D_3 > 0$ such that for any $f, g \in G$ represented by $P$-reduced products $u = u_1 \cdots u_k$ and $v = v_1 \cdots v_m$ respectively, the triangle formed by the paths defined by the products $u = u_1 \cdots u_k$, $v = v_1 \cdots v_m$, and $w = w_1 \cdots w_n$ (representing $f g$) in $\Gamma(G, S)$, is $D_3$-thin.
\end{lemma}
\begin{proof}
Since $w = w_1 \cdots w_n$ represents the element $f g$, the product $w$ can be obtained by reducing the concatenation
$$(u_1 \cdots u_{k-1} u_k) (v_1 v_2 \cdots v_m).$$

\smallskip

{\bf Case I.} $u_k v_1 \notin P$.

Hence, the product
$$u_1 \cdots u_{k-1} u_k v_1 v_2 \cdots v_m$$
is reduced, we have $n = k + m$ and there exist $a_1, \cdots, a_{n-1} \in B_P$ such that 
$$u_1 = w_1 a_1, u_2 = a_1^{-1} w_2 a_2, \cdots, u_k = a_{k-1}^{-1} w_k$$
$$v_1 = a_k^{-1} w_{k+1} a_{k+1}, v_2 = a_{k+1}^{-1} w_{k+2} a_{k+2}, \cdots, v_m = a_{k+m}^{-1} w_n.$$

The following picture in $\Gamma(G, S)$ illustrates the case.

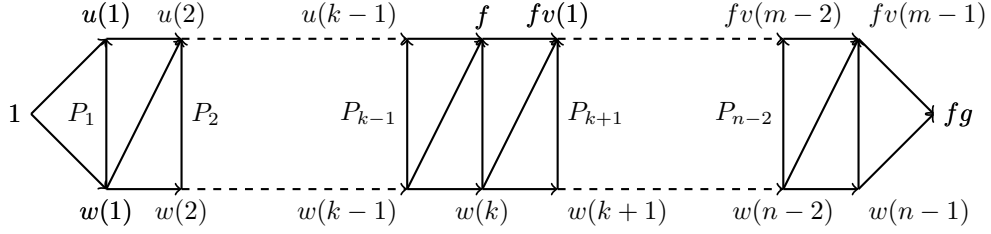
\begin{figure}[H]
\centering
\begin{tikzpicture}
\draw [thick,->] (0,0) node [anchor=east]{$1$} -- (1,1) node [anchor=south]{$u(1)$};
\draw [thick,->] (0,0) node [anchor=east]{$1$} -- (1,-1) node [anchor=north]{$w(1)$};
\draw [thick,->] (1,1) node [anchor=south]{$u(1)$} -- (2,1) node [anchor=south]{$u(2)$};
\draw [thick,->] (1,-1) node [anchor=north]{$w(1)$} -- (2,-1) node [anchor=north]{$w(2)$};
\draw [thick,->][thick,dashed] (2,1) -- (5,1) node [anchor=south east]{$u(k-1)$};
\draw [thick,->][thick,dashed] (2,-1) -- (5,-1) node [anchor=north east]{$w(k-1)$};
\draw [thick,->] (5,1) -- (6,1) node [anchor=south]{$f$};
\draw [thick,->] (5,-1) -- (6,-1) node [anchor=north]{$w(k)$};
\draw [thick,->] (6,1) node [anchor=south]{$f$} -- (7,1) node [anchor=south]{$f v(1)$};
\draw [thick,->] (6,-1) -- (7,-1) node [anchor=north west]{$w(k+1)$};
\draw [thick,->][thick,dashed] (7,1) node [anchor=south]{$f v(1)$} -- (10,1) node [anchor=south]{$f v(m-2)$};
\draw [thick,->] (10,1) -- (11,1) node [anchor=south west]{$f v(m-1)$};
\draw [thick,->][thick,dashed] (7,-1) -- (10,-1) node [anchor=north]{$w(n-2)$};
\draw [thick,->] (10,-1) -- (11,-1) node [anchor=north west]{$w(n-1)$};
\draw [thick,->] (11,1) -- (12,0) node [anchor=west]{$f g$};
\draw [thick,->] (11,-1) -- (12,0) node [anchor=west]{$f g$};
\draw [thick,->] (1,-1) -- node[midway,left]{$P_1$} (1,1);
\draw [thick,->] (2,-1) -- node[midway,right]{$P_2$} (2,1);
\draw [thick,->] (5,-1) -- node[midway,left]{$P_{k-1}$} (5,1);
\draw [thick,->] (6,-1) -- (6,1);
\draw [thick,->] (7,-1) -- node[midway,right]{$P_{k+1}$} (7,1);
\draw [thick,->] (10,-1) -- node[midway,left]{$P_{n-2}$} (10,1);
\draw [thick,->] (11,-1) -- (11,1);
\draw [thick,->] (1,-1) -- (2,1);
\draw [thick,->] (5,-1) -- (6,1);
\draw [thick,->] (6,-1) -- (7,1);
\draw [thick,->] (10,-1) -- (11,1);
\end{tikzpicture}
\caption{$u_k v_1 \notin P$.}
\label{le3-pic1}
\end{figure}

By Lemma \ref{le:reduced-paths}, the Hausdorff distance between the path defined by the product $u_1 \cdots u_{k-1} u_k v_1 v_2 \cdots v_m$ and the path defined by the product $w_1 \cdots w_n$ is bounded by $C_0 + 2C_1$. Hence, the required statement follows for $D_3 = C_0 + 2C_1$.

\smallskip

{\bf Case II.} $u_k v_1 \in P$.

Consider the following possibilities.

\smallskip

{\bf (a)} $u_k v_1 \notin B_p$.

Suppose the product $u_{k-1} (u_k v_1)$ is defined in $P$. Then, from the axiom (P6'), which is equivalent to (P6), it follows that $u_k v_1 \in B_p$, which is a contradiction with the assumption. Hence, $u_{k-1} (u_k v_1)$ is not defined. Similar argument shows that $(u_k v_1) v_2$ is also not defined. Thus, the $P$-product 
$$(u_1 \cdots u_{k-1}) (u_k v_1) (v_2 \cdots v_m)$$
is reduced and it represents the same element (which is $f g$) as the product $w = w_1 \cdots w_n$. Hence, $n = k + m - 1$ and there exist $a_1, \cdots, a_{n-1} \in B_P$ such that
$$u_1 = w_1 a_1, \cdots, u_{k-1} = a_{k-1}^{-1} w_{k-1} a_{k-1}, u_k v_1 = a_{k-1}^{-1} w_k a_k,$$
$$v_2 = a_k^{-1} w_{k+1} a_{k+1}, \cdots, v_{m-1} = a_{k + m - 3}^{-1} w_{k + m - 2} a_{k + m - 2}, v_m = a_{k + m - 2}^{-1} w_{k + m -1}$$
and we have the following picture in $\Gamma(G, S)$.

\begin{figure}[H]
\centering
\begin{tikzpicture}
\draw [thick,->] (0,0) node [anchor=east]{$1$} -- (1,0.5) node [anchor=south]{$u(1)$};
\draw [thick,->] (0,0) node [anchor=east]{$1$} -- (1,-0.5) node [anchor=north]{$w(1)$};
\draw [thick,->] (1,0.5) node [anchor=south]{$u(1)$} -- (2,0.5) node [anchor=south]{$u(2)$};
\draw [thick,->] (1,-0.5) node [anchor=north]{$w(1)$} -- (2,-0.5) node [anchor=north]{$w(2)$};
\draw [thick,->][thick,dashed] (2,0.5) -- (5,0.5) node [anchor=south east]{$u(k-1)$};
\draw [thick,->] (5,0.5) -- (7,0.5);
%\draw [thick,->] (6,0.5) node [anchor=south]{$f$} -- (7,0.5);
\draw [thick,->][thick,dashed] (2,-0.5) -- (5,-0.5) node [anchor=north east]{$w(k-1)$};
\draw [thick,->] (5,-0.5) -- (7,-0.5) node [anchor=north]{$w(k)$};
%\draw [thick,->] (6,-0.5) -- (7,-0.5) node [anchor=north west]{$w(k+1)$};
\draw [thick,->][thick,dashed] (7,0.5) node [anchor=south]{$f v(1)$} -- (10,0.5) node [anchor=south]{$f v(m-2)$};
\draw [thick,->] (10,0.5) -- (11,0.5) node [anchor=south west]{$f v(m-1)$};
\draw [thick,->][thick,dashed] (7,-0.5) -- (10,-0.5) node [anchor=north]{$w(n-2)$};
\draw [thick,->] (10,-0.5) -- (11,-0.5) node [anchor=north west ]{$w(n-1)$};
\draw [thick,->] (11,0.5) -- (12,0) node [anchor=west]{$f g$};
\draw [thick,->] (11,-0.5) -- (12,0) node [anchor=west]{$f g$};
\draw [thick,->] (1,-0.5) -- node [midway,left]{$P_1$} (1,0.5);
\draw [thick,->] (2,-0.5) -- node [midway,right]{$P_2$} (2,0.5);
\draw [thick,->] (5,-0.5) -- node [midway,left]{$P_{k-1}$} (5,0.5);
%\draw [thick,->] (6,-0.5) -- (6,0.5);
\draw [thick,->] (7,-0.5) -- node[midway,right]{$P_k$} (7,0.5);
\draw [thick,->] (10,-0.5) -- node[midway,left]{$P_{n-2}$}  (10,0.5);
\draw [thick,->] (11,-0.5) -- (11,0.5);
\draw [thick,->] (1,-0.5) -- (2,0.5);
\draw [thick,->] (5,-0.5) -- (7,0.5);
%\draw [thick,->] (6,-0.5) -- (7,0.5);
\draw [thick,->] (10,-0.5) -- (11,0.5);
\draw [thick,->] (5,0.5) -- (6,1.5) node [anchor=south]{$u(k) = f$};
\draw [thick,->] (6,1.5) -- (7,0.5);
\end{tikzpicture}
\caption{$u_k v_1 \in P \smallsetminus B_P$.}
\label{le3-pic2}
\end{figure}
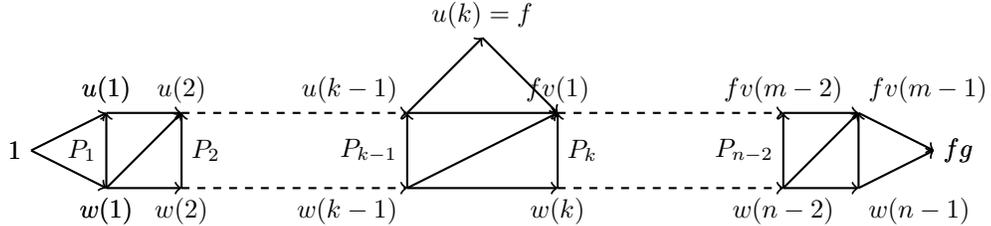

Take a point $t$ on the path defined by the product $u$. If $t$ belongs to the path defined by $u(k-1)$, then using the axioms (H0) and (H1), we find a point $t'$ on the path defined by $w$ such that $d(t, t') \leqslant C_0 + 2C_1$. If $t \in [u(k-1), f]$, then since the geodesic triangle $\Delta(u(k-1), f, f v(1))$ is $C_1$-thin, there exists a point $s \in [u(k-1), f v(1)] \cup [f, f v(1)]$ such that $d(t, s) \leqslant C_1$. If $s \in [f, f v(1)]$, then we are done. If $s \in [u(k-1), f v(1)]$, then by Lemma \ref{le:reduced-paths} there exists a point $s' \in [w(k-1), w(k)]$ such that $d(s, s') \leqslant C_0 + 2C_1$ and
$$d(t, s') \leqslant d(t, s) + d(s, s') \leqslant C_0 + 3C_1.$$

Using a similar argument, we obtain the same bound in the case when $t$ belongs to the path defined by $v$ (translated to $f$) and the path defined by the product $w$.

\smallskip

{\bf (b)} $u_k v_1 \in B_p$.

It follows that $u_{k-1} (u_k v_1)$ and $(u_k v_1) v_2$ are both defined in $P$. Consider the following cases.

\smallskip

{\bf (1)} $(u_{k-1} (u_k v_1)) v_2$ is not defined in $P$.

It follows that the product 
$$(u_1 \cdots u_{k-2}) (u_{k-1} (u_k v_1)) (v_2 \cdots v_m)$$
is reduced and we obtain $n = k + m - 2$. Again, there exist $a_1, \cdots, a_n \in B_P$, using which one can interleave the above reduced product and obtain the product $w$.

Using the argument similar to {\bf (a)}, we obtain the required bound.

\smallskip

{\bf (2)} $(u_{k-1} (u_k v_1)) v_2$ is defined in $P$.

From the axiom (P6) it follows that $(u_{k-1} (u_k v_1)) v_2 \in B_P$ and the reduction process continues. Let us assume that 
$$(u_{r+1} \cdots u_{k-1}) (u_k v_1) (v_2 \cdots v_{s-1}) \in P \smallsetminus B_P$$
for some $r$ and $s$ (in this case $k - r = s - 1$). Hence, the product
$$(u_1 \cdots u_r) (u_{r+1} \cdots u_{k-1} u_k v_1 v_2 \cdots v_{s-1}) (v_s \cdots v_m)$$
is reduced and we have the following picture.

\begin{figure}[H]
\centering
\begin{tikzpicture}
\draw [thick,->] (0,-0.25) node [anchor=east]{$1$} -- (1,0.5) node [anchor=south]{$u_1$};
\draw [thick,->] (1,0.5) node [anchor=south]{$u_1$} -- (2,0.5) node [anchor=south]{$u(2)$};
\draw [thick,->] (1,-1) node [anchor=north]{$w_1$} -- (2,-1) node [anchor=north]{$w(2)$};
\draw [thick,->] (0,-0.25) node [anchor=east]{$1$} -- (1,-1) node [anchor=north]{$w_1$};
\draw [thick,->,dashed] (2,0.5) -- (4,0.5);
\draw [thick,->] (4,0.5) node [anchor=south]{$u(r-1)$} -- (5,0.5);
\draw [thick,->] (5,0.5) node [anchor=north west]{$u(r)$} -- (5,1.5) node [anchor=east]{$u(r+1)$};
\draw [thick,->,dashed] (5,1.5) -- (5,3) node [anchor= east]{$u(k-2)$};
\draw [thick,->] (5,3) node [anchor= east]{$u(k-2)$} -- (5,4) node [anchor=east]{$u(k-1)$};
\draw [thick,->] (5,4) -- (6,5) node [anchor=south]{$f$};
\draw [thick,->] (6,5) -- (7,4) node [anchor=west]{$f v_1$};
\draw [thick,->] (7,4) -- (7,3) node [anchor=west]{$f v(2)$};
\draw [thick,->,dashed] (7,3) -- (7,1.5) node [anchor=west]{$f v(s-1)$};
\draw [thick,->] (7,1.5) -- (7,0.5) node [anchor=south west]{$f v(s)$};
\draw [thick,->] (7,0.5) -- (8,0.5);
\draw [thick,->,dashed] (8,0.5) -- (10,0.5) node [anchor=south ]{$f v(m-2)$};
\draw [thick,->] (10,0.5) -- (11,0.5) node [anchor=south west]{$f v(m-1)$};
\draw [thick,->] (11,0.5) -- (12,-0.25);
\draw [thick,->,dashed] (7,-1) node [anchor=north east]{} -- (10,-1) node [anchor=north]{$w(n-2)$};
\draw [thick,->] (4,-1) node [anchor=north]{$w(r-1)$} -- (5,-1);
\draw [thick,->] (5,-1) node [anchor=north west]{$w(r)$} -- (7,-1) node [anchor=north]{$w(k+s)$};
\draw [thick,->] (7,-1) -- (8,-1);
\draw [thick,->] (7,-1) -- (7,0.5);
\draw [thick,->] (5,-1) -- (5,0.5);
\draw [thick,->] (5,-1) -- (7,0.5);
\draw [thick,->] (5,0.5) -- (7,0.5);
\draw [thick,->,dashed] (2,-1) -- (4,-1);
\draw [thick,->] (10,-1) -- (11,-1) node [anchor=north west]{$w(n-1)$};
\draw [thick,->] (11,-1) node [anchor=north]{} -- (12,-0.25) node [anchor=west]{$f g$};
\draw [thick,->] (1,-1) -- node [midway,left]{$P_1$} (1,0.5);
\draw [thick,->] (2,-1) -- node[midway,right]{$P_2$} (2,0.5);
\draw [thick,->] (10,-1) -- (10,0.5);
\draw [thick,->] (11,-1) -- (11,0.5);
\draw [thick,->] (1,-1) -- (2,0.5);
\draw [thick,->] (10,-1) -- (11,0.5);
\draw [thick,->] (5,3) -- (7,3);
\draw [thick,->] (5,4) -- (7,4);
\draw [thick,->] (5,1.5) -- (7,1.5);
\draw [thick,->] (5,3) -- (7,4);
\draw [thick,->] (5,0.5) -- (7,1.5);
\draw [thick,->] (4,-1) -- node [midway,left]{$P_{r-1}$} (4,0.5);
\draw [thick,->] (4,-1) -- (5,0.5);
\draw [thick,->] (8,-1) -- node [midway,right] {$P_{k+s+1}$} (8,0.5);
\draw [thick,->] (7,-1) -- (8,0.5);
\end{tikzpicture}
\caption{The general case.}
\label{le3-pic3}
\end{figure}
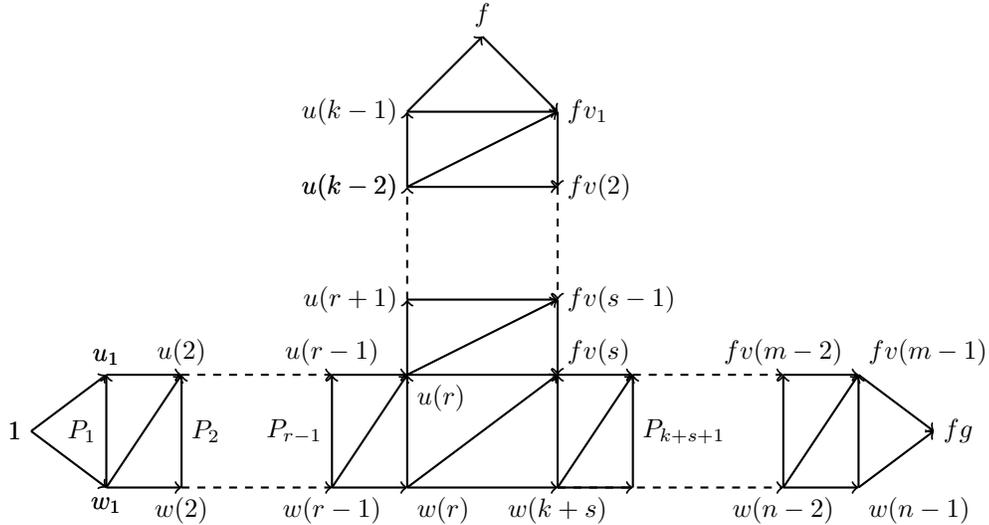

The argument in the general case is similar to {\bf (a)} and we obtain the same bound $D_3 = C_0 + 3C_1$.
\end{proof}

The above lemma can be generalized to the following result.

\begin{lemma} \label{le:thin-quads}
Let $P$ satisfy the axioms (P6) and (H0)-(H3). Then there exists $D_4 > 0$ such that for any reduced $P$-product $p = p_1 \cdots p_n$ and any point $x \in [1, p_1]$, the Hausdorff distance between the path
$$\mathcal{P} = [x, p(1)] \cup [p(1), p(2)] \cup \cdots \cup [p(n-1), p(n)]$$
and the path
$$\mathcal{U} = [x, x u(1)] \cup [x u(1), x u(2)] \cup \cdots \cup [x u(k-1), x u(k)]$$
defined by a reduced $P$-product $u = u_1 \cdots u_k$, representing $x^{-1} p$, is bounded by $D_4$.
\end{lemma}
\begin{proof}
Note that since $x \in [1, p_1]$, by the axiom (H3), $p_1^{-1} x = c_1 c_2$, where $|c_2| \leqslant C_3$. Hence, the path
$$[x, p(1)] \cup [p(1), p(2)] \cup \cdots \cup [p(n-1), p(n)]$$
corresponds to the $P$-product
$$(c_2^{-1} c_1^{-1}) (p_2 \cdots p_n),$$
which may not be reduced, and which represents the element $x^{-1} p$. The same element is represented by the reduced $P$-product $u = u_1 \cdots u_k$, that is,
$$(c_2^{-1} c_1^{-1}) (p_2 \cdots p_n) = u_1 \cdots u_k$$
and we analyze the reduction process on the left-hand side. Consider the following cases.

\smallskip

{\bf Case I.} $p_1 c_1 \notin P$.

From the axiom (H2) it follows that $(p_1 c_1 \cdot 1)_{p_1} \leqslant C_2$ and we have
$$C_2 \geqslant (p_1 c_1 \cdot 1)_{p_1} \geqslant (p_1 c_1 \cdot x)_{p_1} = (p_1 c_1 \cdot p_1 c_1 c_2)_{p_1}$$
$$= \frac{1}{2} (d(p_1, p_1 c_1) + d(p_1, p_1 c_1 c_2) - d(p_1 c_1, p_1 c_1 c_2))$$
$$= \frac{1}{2} (|c_1| + d(x, p_1) - |c_2|) \geqslant \frac{1}{2} (|c_1| + d(x, p_1) - C_3).$$
That is,
$$|c_1| + d(x, p_1) \leqslant 2C_2 + C_3,$$
which means that both $|c_1|$ and $d(x, p_1)$ are bounded. Hence, any geodesic triangle $\Delta(x, p_1, p_1 c_1)$ has a bounded diameter. In particular, for every $t \in [x, p_1]$ there is a point $s \in [x, p_1 c_1] \cup [p_1 c_1, p_1]$ such that $d(t, s) \leqslant C_1 + C_3$.

\smallskip

{\bf (a)} $c_1^{-1} p_2 \notin P$.

In this case the product
$$c_2^{-1} c_1^{-1} p_2 \cdots p_n$$
is reduced. By Lemma \ref{le:reduced-paths}, the Hausdorff distance between the path 
$$[x, p_1 c_1] \cup [p_1 c_1, p(1)] \cup [p(1), p(2)] \cup \cdots \cup [p(n-1), p(n)]$$
and $\mathcal{U}$ is bounded by a constant $D_2$. The Hausdorff distance between $[x, p_1]$ and $[x, p_1 c_1] \cup [p_1 c_1, p(1)]$ is bounded by $C_1 + C_3$. Hence, we can take $D_4 = D_3 + C_1 + C_3$.

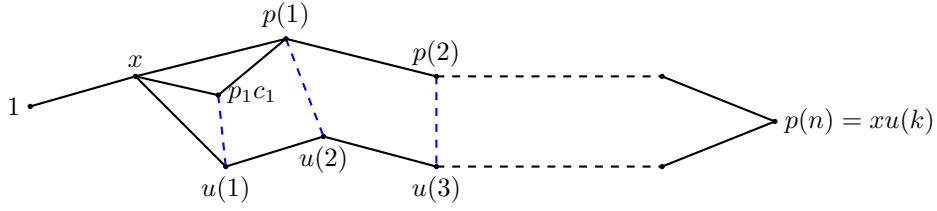
\begin{figure}[H]
\centering
\begin{tikzpicture}
\draw [thick] (0.6,0.1) node [anchor=east]{$1$} --(2,0.5)node [anchor=south]{$x$};
\draw [thick] (2,0.5)  --(4,1)node [anchor=south]{$p(1)$};
\draw [thick] (2,0.5)  --(3.1,0.25);
\draw [thick] (3.1,0.25) node [anchor=west]{$p_1c_1$} --(4,1);
\draw [thick] (4,1)  --(6,0.5)node [anchor=south]{$p(2)$};
\draw [thick,dashed] (6,0.5) --(9,0.5);
\draw [thick] (9,0.5)  --(10.5,-0.1)node [anchor= west]{$p(n)=xu(k)$};
\draw [thick] (2,0.5) --(3.2,-0.7)node [anchor=north]{$u(1)$};
\draw [thick] (3.2,-0.7)  --(4.5,-0.3)node [anchor=north]{$u(2)$};
\draw [thick] (4.5,-0.3)  --(6,-0.7)node [anchor=north]{$u(3)$};
\draw [thick,dashed] (6,-0.7)  --(9,-0.7);
\draw [thick] (9,-0.7)  --(10.5,-0.1);
\draw [thick,dashed,blue] (3.1,0.25)  --(3.2,-0.7);
\draw [thick,dashed,blue] (4,1)  --(4.5,-0.3);
\draw [thick,dashed,blue] (6,0.5)  --(6,-0.7);
\filldraw[black] (0.6,0.1) circle (0.8 pt);
\filldraw[black] (2,0.5) circle (0.8 pt);
\filldraw[black] (4,1) circle (0.8 pt);
\filldraw[black] (3.1,0.25) circle (0.8 pt);
\filldraw[black] (6,0.5) circle (0.8 pt);
\filldraw[black] (9,0.5) circle (0.8 pt);
\filldraw[black] (10.5,-0.1) circle (0.8 pt);
\filldraw[black] (3.2,-0.7) circle (0.8 pt);
\filldraw[black] (4.5,-0.3) circle (0.8 pt);
\filldraw[black] (6,-0.7) circle (0.8 pt);
\filldraw[black] (9,-0.7) circle (0.8 pt);
\end{tikzpicture}
\caption{$c_1^{-1} p_2 \notin P$.}
\label{le4-pic1}
\end{figure}

\smallskip

{\bf (b)} $c_1^{-1} p_2 \in P$, but $c_2^{-1} (c_1^{-1} p_2) \notin P$.

In this case the product
$$c_2^{-1} (c_1^{-1} p_2) p_3 \cdots p_n$$
is reduced. By Lemma \ref{le:reduced-paths}, the Hausdorff distance between the path 
$$[x, p_1 c_1] \cup [p_1 c_1, p(2)] \cup [p(2), p(3)] \cup \cdots \cup [p(n-1), p(n)]$$
and $\mathcal{U}$ is bounded by a constant $D_2$. At the same time, the geodesic triangle $\Delta(p_1 c_1, p_1, p(2))$ is $C_2$-thin and $d(p_1, p_1 c_1) \leqslant 2C_2 + C_3$. Hence, the Hausdorff distance between $[x, p_1] \cup [p_1, p(2)]$ and $[x, p_1 c_1] \cup [p_1 c_1, p(2)]$ is bounded by $C_1 + 2C_2 + C_3$. 
Hence, we can take $D_4 = D_3 + C_1 + 2C_2 + C_3$.

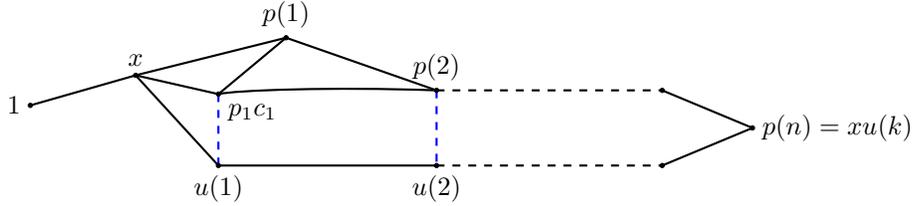
\begin{figure}[H]
\centering
\begin{tikzpicture}
\draw [thick] (0.6,0.1) node [anchor=east]{$1$} --(2,0.5)node [anchor=south]{$x$};
\draw [thick] (2,0.5)  --(4,1)node [anchor=south]{$p(1)$};
\draw [thick] (2,0.5)  --(3.1,0.25);
\draw [thick] (3.1,0.25) node [anchor=north west]{$p_1c_1$} --(4,1);
\draw [thick] (4,1)  --(6,0.3)node [anchor=south]{$p(2)$};
\draw [thick,dashed] (6,0.3) --(9,0.3);
\draw [thick] (9,0.3)  --(10.2,-0.2)node [anchor= west]{$p(n)=xu(k)$};
\draw [thick] (2,0.5) --(3.1,-0.7)node [anchor=north]{$u(1)$};
\draw [thick] (3.1,-0.7)  --(6,-0.7)node [anchor=north]{$u(2)$};
\draw [thick,dashed] (6,-0.7)  --(9,-0.7);
\draw [thick] (9,-0.7)  --(10.2,-0.2);
\draw [thick] (3.1,0.25)  ..controls(3.3,0.3) and (4.6,0.35)..(6,0.3);
\draw [thick,dashed,blue] (3.1,0.25)  --(3.1,-0.7);
\draw [thick,dashed,blue] (6,0.3)  --(6,-0.7);
\filldraw[black] (0.6,0.1) circle (0.8 pt);
\filldraw[black] (2,0.5) circle (0.8 pt);
\filldraw[black] (4,1) circle (0.8 pt);
\filldraw[black] (3.1,0.25) circle (0.8 pt);
\filldraw[black] (6,0.3) circle (0.8 pt);
\filldraw[black] (9,0.3) circle (0.8 pt);
\filldraw[black] (10.2,-0.2) circle (0.8 pt);
\filldraw[black] (3.1,-0.7) circle (0.8 pt);
\filldraw[black] (6,-0.7) circle (0.8 pt);
\filldraw[black] (9,-0.7) circle (0.8 pt);
\end{tikzpicture}
\caption{$c_1^{-1} p_2 \in P$, but $c_2^{-1} (c_1^{-1} p_2) \notin P$.}
\label{le4-pic2}
\end{figure}

\smallskip

{\bf (c)} $c_1^{-1} p_2 \in P$ and $c_2^{-1} (c_1^{-1} p_2) \in P$.

From the axiom (P6) it follows that $c_1^{-1} p_2 \in B_P$, hence, by the axiom (H0) we have $d(p_1 c_1, p(2)) \leqslant C_0$. It follows that $|p_2| \leqslant C_0 + 2C_2 + C_3$ an the geodesic triangles $\Delta(p_1 c_1, p_1, p(2))$ and $\Delta(x, p_1 c_1, p(2))$ have bounded diameters. Regardless of whether $(c_2^{-1} (c_1^{-1} p_2)) p_3$ is defined in $P$, we obtain a bounded Hausdorff distance between $\mathcal{P}$ and $\mathcal{U}$.

\begin{figure}[H]
\centering
\begin{tikzpicture}
\draw [thick] (0.6,0.1) node [anchor=east]{$1$} --(2,0.5)node [anchor=south]{$x$};
\draw [thick] (2,0.5)  --(4,1)node [anchor=south]{$p(1)$};
\draw [thick] (2,0.5) ..controls(3.3,0.4) and (4.6,0.4).. (6,0.5);
\draw [thick] (4,1)  --(6,0.5)node [anchor=south]{$p(2)$};
\draw [thick,dashed] (6,0.5) --(9,0.5);
\draw [thick] (9,0.5)  --(10.5,-0.1)node [anchor= west]{$p(n)=xu(k)$};
\draw [thick] (2,0.5)  ..controls(2.5,0.1) and (3,-0.3)..(6,-0.7)node [anchor=north]{$u(1)$};
\draw [thick,dashed] (6,-0.7)  --(9,-0.7);
\draw [thick] (9,-0.7)  --(10.5,-0.1);
\draw [thick,dashed,blue] (6,0.5)  --(6,-0.7);
\filldraw[black] (0.6,0.1) circle (0.8 pt);
\filldraw[black] (2,0.5) circle (0.8 pt);
\filldraw[black] (4,1) circle (0.8 pt);
\filldraw[black] (6,0.5) circle (0.8 pt);
\filldraw[black] (9,0.5) circle (0.8 pt);
\filldraw[black] (10.5,-0.1) circle (0.8 pt);
\filldraw[black] (6,-0.7) circle (0.8 pt);
\filldraw[black] (9,-0.7) circle (0.8 pt);
\end{tikzpicture}
\caption{$c_1^{-1} p_2 \in P$ and $c_2^{-1} (c_1^{-1} p_2) \in P$.}
\label{le4-pic3}
\end{figure}
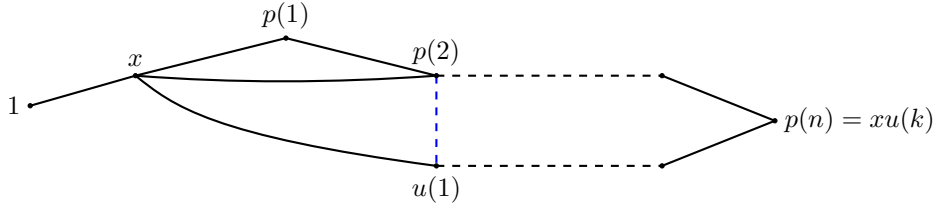

\smallskip

{\bf Case II.} $p_1 c_1 \in P$.

Consider the following possibilities, which are similar to the ones in {\bf Case I.}

\smallskip

{\bf (a)} $c_1^{-1} p_2 \notin P$.

The product
$$c_2^{-1} c_1^{-1} p_2 \cdots p_n$$
is reduced. By Lemma \ref{le:reduced-paths}, the Hausdorff distance between the path 
$$[x, p_1 c_1] \cup [p_1 c_1, p(1)] \cup [p(1), p(2)] \cup \cdots \cup [p(n-1), p(n)]$$
and $\mathcal{U}$ is bounded by a constant $D_2$. The Hausdorff distance between $[x, p_1]$ and $[x, p_1 c_1] \cup [p_1 c_1, p(1)]$ is bounded by $C_1 + C_3$. Hence, we can take $D_4 = D_3 + C_1 + C_3$.

\smallskip

{\bf (b)} $c_1^{-1} p_2 \in P$.

Since $p_1 p_2$ is reduced, but both $p_1 c_1 \in P$ and $c_1^{-1} p_2 \in P$, from the axiom (P6) it follows that $c_1 \in B_P$. Hence, $c_1 c_2 = q \in P$. Now, considering the cases $q p_2 \in P$ and $q p_2 \notin P$ and using Lemma \ref{le:reduced-paths} we obtain the bounded Hausdorff distance between $\mathcal{P}$ and $\mathcal{U}$.
\end{proof}

Finally, the following lemma is crucial for the proof of the main theorem. For a geodesic $\gamma$ in $\Gamma(G, S)$ and two points $x, y$ on $\gamma$, by $\gamma(x, y)$ we denote the segment of $\gamma$ between $x$ and $y$.

\begin{lemma} \label{le:polygons}
Let $P$ satisfy the axioms (P6) and (H0)-(H3). Then there exists $D_5 > 0$ such that for every geodesic $\gamma$ connecting an element $p \in P \smallsetminus B_P$ to the identity in $\Gamma(G, S)$ and any points $x, y \in \gamma$, the Hausdorff distance between $\gamma(x, y)$ and the path 
$$\mathcal{C} = [x, x c(1)] \cup [x c(1), x c(2)] \cup \cdots \cup [x c(k-1), x c(k)]$$
defined by any reduced $P$-product $c_1 \cdots c_k$ representing $x^{-1} y$, is bounded by $D_5$.
\end{lemma}
\begin{proof}
If $x = 1$ and $y = p$, then $\gamma(x, y) = \gamma$ and the Hausdorff distance between $\gamma$ and any other geodesic connecting $1$ and $p$ is bounded by $\delta$, which follows from (H1).

\smallskip

Suppose, $x = 1$ and $y$ is a point on $\gamma$. From (H3) it follows that $y = p_1 p_2$, which is a reduced $P$-product. Hence, any geodesic triangle $\Delta(1, p_1, p_1 p_2) = [1, p_1] \cup [p_1, p_1 p_2] \cup [1, p_1 p_2]$ on the vertices $\{1, p_1, p_1 p_2\}$ is $C_1$-thin, by the axiom (H1), and for every $t_0 \in [1, p_1 p_2]$ there exists $t_1 \in [1, p_1] \cup [p_1, p_1 p_2]$ such that $d(t_0, t_1) \leqslant C_1$. At the same time, from (H3) it also follows that $|p_2| \leqslant C_3$, so for every point $z_0 \in [1, p_1] \cup [p_1, p_1 p_2]$ there exists a point $z_1 \in [1, p_1 p_2]$ such that $d(z_0, z_1) \leqslant C_1 + C_3$. The axiom (H2) can also be applied here: the Gromov product $(x \cdot y)_{p_1}$ is bounded by $C_2$ since the product $p_1 p_2$ is $P$-reduced and the distance from any $z_0 \in [1, p_1] \cup [p_1, p_1 p_2]$ to the geodesic $[1, p_1 p_2]$ is bounded by $C_1 + C_2$.

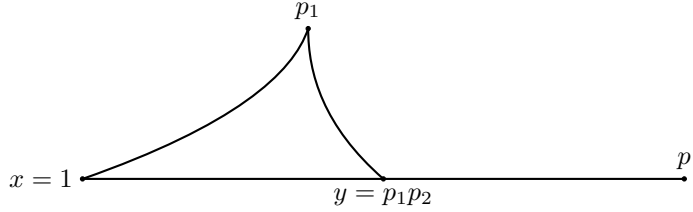
\begin{figure}[H]
\centering
\begin{tikzpicture}
\draw [thick] (0,0) node [anchor=east]{$x = 1$} -- (8,0) node [anchor=south]{$p$};
\draw [thick] (0,0) .. controls (2,0.7) and (2.8,1.4) .. (3,2) node [anchor=south]{$p_1$};
\draw [thick](4,0) node [anchor=north]{$y = p_1 p_2$} .. controls (3.2,0.7) and (3,1.4) .. (3,2);
\filldraw [black] (0,0) circle (0.8 pt);
\filldraw [black] (8,0) circle (0.8 pt);
\filldraw [black] (3,2) circle (0.8 pt);
\filldraw [black] (4,0) circle (0.8 pt);
\end{tikzpicture}
\caption{$x = 1$ and $y$ is a point on $\gamma$.}
\label{le4-pic4}
\end{figure}

If $y = p$ and $x$ is a point on $\gamma$, then the argument above works by symmetry.

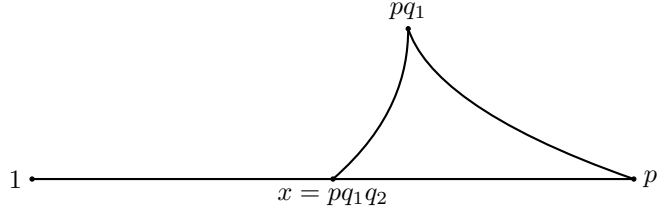
\begin{figure}[H]
\centering
\begin{tikzpicture}
\draw [thick] (0,0) node [anchor=east]{$1$} -- (8,0) node [anchor= west]{$p$};
\draw [thick] (4,0) node [anchor=north]{$x = p q_1 q_2$} .. controls (4.8,0.7) and (5,1.4) .. (5,2) node [anchor=south]{$p  q_1$};
\draw [thick] (8,0) .. controls (6,0.7) and (5.2,1.4) .. (5,2);
\filldraw [black] (0,0) circle (0.8 pt);
\filldraw [black] (8,0) circle (0.8 pt);
\filldraw [black] (5,2) circle (0.8 pt);
\filldraw [black] (4,0) circle (0.8 pt);
\end{tikzpicture}
\caption{$y = p$ and $x$ is a point on $\gamma$.}
\label{le4-pic5}
\end{figure}

\smallskip

Now suppose $x, y \in \gamma$ are arbitrary points (without loss of generality we assume that $x \in \gamma(1, y)$). Note that by the axiom (H3), $x = p_1 p_2$, where $|p_2| \leqslant C_3$, and, similarly, $p^{-1} y = q_1 q_2$, where $|q_2| \leqslant C_3$.

Note that $x^{-1} y \in G$, hence, $x^{-1} y = c_1 \cdots c_k$, where $c_i \in P$ and the product on the right is $P$-reduced. Recall that 
$$\mathcal{C} = [x, x c(1)] \cup [x c(1), x c(2)] \cup \cdots \cup [x c(k-1), x c(k)]$$
be the path in $\Gamma(G, S)$ defined by $c_1 \cdots c_k$, where $[x c(i), c(i+1)]$ is an arbitrary geodesic connecting $x c(i)$ and $x c(i+1)$ for $i \in [0, k - 1]$.

\smallskip

{\bf Case I.} If $k = 1$, or $k = 2$, then the Hausdorff distance between $\mathcal{C}$ and $\gamma(x, y)$ is bounded by $C_1 + C_2$, which follows from the axioms (H1) and (H2).
\begin{figure}[H]
\centering
\begin{tikzpicture}
\draw [thick] (0,0) node [anchor=east]{$1$} -- (9,0) node [anchor= west]{$p$};
\draw [thick] (0,0) .. controls (2,0.7) and (2.4,1.4) .. (2.5,2) node [anchor=south]{$p_1$};
\draw [thick] (2.5,2) .. controls (2.5,1.4) and (2.6,0.7) .. (3,0) node [anchor=north]{$x=p_1p_2$};
\draw [thick] (3,0) .. controls (3.8,0.7) and (4.4,1.4) .. (4.5,2) node [anchor=south]{$x c_1$};
\draw [thick] (4.5,2) .. controls (4.6,1.4) and (5.2,0.7) .. (6,0);
\draw [thick](6,0) node [anchor=north]{$y = x c_1 c_2 = p q_1 q_2$} .. controls (6.4,0.7) and (6.5,1.4) .. (6.5,2) node [anchor=south]{$pq_1$};
\draw [thick] (6.5,2) .. controls (6.6,1.4) and (7,0.7) .. (9,0);
\filldraw [black] (0,0) circle (0.8 pt);
\filldraw [black] (9,0) circle (0.8 pt);
\filldraw [black] (4.5,2) circle (0.8 pt);
\filldraw [black] (2.5,2) circle (0.8 pt);
\filldraw [black] (6.5,2) circle (0.8 pt);
\filldraw [black] (3,0) circle (0.8 pt);
\filldraw [black] (6,0) circle (0.8 pt);
\end{tikzpicture}
\caption{$k = 2$.}
\label{le4-pic6}
\end{figure}

\smallskip

{\bf Case II.} Let $k = 3$. That is, $x^{-1} y = c_1 c_2 c_3$.

In this case we obtain $p = (p_1 p_2) (c_1 c_2 c_3) (q_2^{-1} q_1^{-1})$, which is possible only if $p_2 c_1 \in P$ and $c_3 q_2^{-1} \in P$. Moreover, if either $p_1 (p_2 c_1) \notin P$, or $(c_3 q_2^{-1}) q_1^{-1} \notin P$, then we obtain a contradiction since $p$ can be represented by a reduced $P$-product of length more than one. Hence, $p_1 (p_2 c_1) \in P$ and $(c_3 q_2^{-1}) q_1^{-1} \in P$, and from the axiom (P6) it follows that $p_2 c_1, c_3 q_2^{-1} \in B_P$.
In this case any geodesic triangles $\Delta(p_1, x, x c_1)$ and $\Delta(p q_1, y, y c_3^{-1})$ have bounded diameters. Without loss of generality we can assume that $p_2 c_1 = c_3 q_2^{-1} = 1$. Hence, $p = p_1 c_2 q_1^{-1}$ and the product on the right is not reduced.

\smallskip

{\bf (a)} Assume that $p_1^{-1} p \notin P$. Then from the axiom (H2) it follows that $(p_1 \cdot p)_1 \leqslant C_2$ and we have
$$C_2 \geqslant (p_1 \cdot p)_1 \geqslant (p_1 \cdot x)_1 = (p_1 \cdot p_1 p_2)_1$$
$$= \frac{1}{2} (d(1, p_1) + d(1, p_1 p_2) - d(p_1, p_1 p_2))$$
$$= \frac{1}{2} (|p_1| + d(1, p_1 p_2) - |p_2|) \geqslant \frac{1}{2} (|p_1| + d(1, x) - C_3).$$
That is,
$$|p_1| + d(1, x) \leqslant 2C_2 + C_3,$$
which means that both $|p_1|$ and $d(1, x)$ are bounded. Hence, any geodesic triangle $\Delta(1, p_1, p_1 p_2)$ has a bounded diameter.

Since $p = p_1 c_2 q_1^{-1}$ and $p_1^{-1} p \notin P$, it follows that $p_1^{-1} p = c_2 q_1^{-1}$ and both products are reduced. But then, $p_1^{-1} a = c_2$ for some $a \in B_P$, and it follows that $|c_2| \leqslant |p_1| + |a| \leqslant (2C_2 + C_3) + C_0$ and the quadrilateral 
$$[x, x c_1] \cup [x c_1, x c_1 c_2] \cup [x c_1 c_2, y] \cup \gamma(x, y)$$
has a bounded diameter. Hence, the Hausdorff distance between $\gamma(x, y)$ and $\mathcal{C} = [x, x c_1] \cup [x c_1, x c_1 c_2] \cup [x c_1 c_2, y]$ is bounded.

\smallskip

{\bf (b)} If $p q_1 \notin P$, then an argument similar to {\bf (a)} can be applied.

\smallskip

{\bf (c)} Suppose both $p_1^{-1} p \in P$ and $p q_1 \in P$.

Let $t \in \gamma(x, y)$. Since the geodesic triangle $\Delta(x, p_1, p)$ is $C_1$-thin and $d(x, p_1) \leqslant C_3$, there exists $s \in [p_1, p]$ such that $d(s, t) \leqslant C_1 + C_3$. 

If $d(p_1, s) \leqslant (p \cdot p_1 c_2)_{p_1}$, then $s$ is $C_1$ close to $[p_1, p_1 c_2] = [x c_1, x c_1 c_2]$. 

Suppose $d(p_1, s) > (p \cdot p_1 c_2)_{p_1}$. Hence, $d(p, s) \leqslant (p_1 \cdot p_1 c_2)_p$ and there exists $s_1 \in [p, p_1 c_2]$ such that $d(s, s_1) \leqslant C_1$. It follows that $d(t, s_1) \leqslant d(t, s) + d(s, s_1) \leqslant 2C_1 + C_3$. At the same time, the triangle $\Delta(y, p, p_1 c_2)$ is $C_1$-thin, so there exists $t_1 \in [y, p_1 c_2] \cup [y, p]$ such that $d(s_1, t_1) \leqslant C_1$, and in this case $d(t, t_1) \leqslant d(t, s_1) + d(s_1, t_1) \leqslant 3C_1 + C_3$. Thus, $d(t, y) \leqslant 3C_1 + 2C_3$ and it follows that $d(t, p_1 c_2) \leqslant 3C_1 + 3C_3$.

Now take a point $z \in [x, x c_1] \cup [x c_1, x c_1 c_2] \cup [x c_1 c_2, y]$. If $z \in [x, x c_1] \cup [x c_1 c_2, y]$, then there exists a point on $\gamma(x, y)$ (for example, $x$ or $y$) at the distance not more than $C_3$ from $z$. Suppose, $z \in [x c_1, x c_1 c_2] = [p_1, p_1 c_2]$. Note that any geodesic triangle $\Delta(p_1, p_1 c_2, p)$ is $C_1$-thin (by (H1)), hence, there exists $z_1 \in [p_1, p] \cup [p_1 c_2, p]$ such that $d(z, z_1) \leqslant C_1$.

\begin{figure}[H]
\centering
\begin{tikzpicture}
\draw [thick] (0,0) node [anchor=east]{$1$} -- (8,0) node [anchor= west]{$p$};
\draw [thick] (2,2) -- (8,0) node [anchor=west]{$p$};
\draw [thick] (0,0) .. controls (1.4,0.7) and (1.8,1.4) .. (2,2) node [anchor=south]{$p_1 = x c_1$};
\draw [thick] (2,2) -- (6,2) node [anchor=south]{$p_1 c_2 = x c_1 c_2 = p q_1$};
\draw [thick] (2,2) -- (2,0) node [anchor=north]{$x$};
\draw [thick] (6,2) -- (6,0) node [anchor=north]{$y$};
\draw [thick,dashed] (4.4,1.2) node [anchor=south]{$s$} .. controls (4.5,0.6) .. (4.4,0) node [anchor= north]{$t$};
\draw [thick](6,2) .. controls (6.05,1.8) and (6.3,1.5) .. (6.5,1.2) node [anchor=west]{$s_1$};
\draw [thick] (6.5,1.2) .. controls (6.8,0.8) and (7.17,0.41) .. (8,0);
\draw [thick,dashed] (4.4,1.2) node [anchor=south]{$s$} .. controls (5.45,1.1) .. (6.5,1.2);
\draw [thick,dashed] (6.5,1.2) .. controls (6.42,0.6) .. (6.5,0) node [anchor=north]{$t_1$};
\filldraw [black] (4.4,1.2) circle (0.8 pt);
\filldraw [black] (0,0) circle (0.8 pt);
\filldraw [black] (8,0) circle (0.8 pt);
\filldraw [black] (2,2) circle (0.8 pt);
\filldraw [black] (6,2) circle (0.8 pt);
\filldraw [black] (2,0) circle (0.8 pt);
\filldraw [black] (6,0)  circle (0.8 pt);
\filldraw [black] (6.5,1.2) circle (0.8 pt);
\filldraw [black] (6.5,0) circle (0.8 pt);
\filldraw [black] (4.4,0) circle (0.8 pt);
\end{tikzpicture}
\caption{(c) $p_1^{-1} p \in P$ and $p q_1 \in P$.}
\label{le4-pic7}
\end{figure}
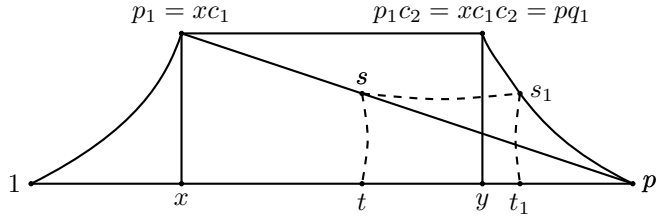

\smallskip

{\bf (i)} Let $z_1 \in [p_1, p]$. Since the triangle $\Delta(p_1, x, p)$ is $C_1$-thin, there exists $w_1 \in \gamma(x, p)$ such that $d(z_1, w_1) \leqslant C_1$. If $w_1 \in \gamma(x, y)$, then $d(z, w_1) \leqslant d(z, z_1) + d(z_1, w_1) \leqslant 2C_1$. Suppose $w_1 \in \gamma(y, p)$. Hence, $d(z_1, p) = d(w_1, p) \leqslant d(y, p)$ and there exists $y' \in [p_1, p]$ such that $d(y', p) = d(y, p) \geqslant d(z_1, p)$ and $d(y, y') \leqslant C_1$ (since $\Delta(p_1, p, x)$ is $C_1$-thin). But then it follows that there exists $y'' \in [p_1, p_1 c_2]$ such that $d(y', y'') \leqslant C_1$ (since $\Delta(p_1, p_1 c_2, p)$ is $C_1$-thin). Then, we obtain
$$d(p_1 c_2, y'') \leqslant d(p_1 c_2, y) + d(y, y') + d(y', y'') \leqslant C_3 + C_1 + C_1 = 2C_1 + C_3.$$
But $d(y', p_1) = d(y'', p_1) \leqslant d(z, p_1)$. Hence,
$$d(z, p_1 c_2) \leqslant d(y'', p_1 c_2) \leqslant 2C_1 + C_3$$
and we finally obtain $d(z, y) \leqslant 2C_1 + 2C_3$.

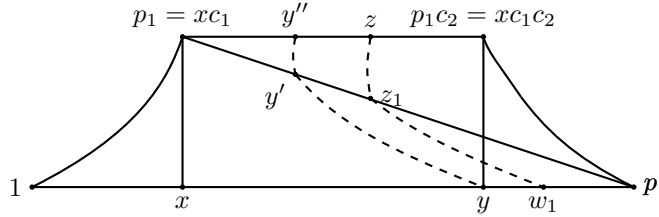
\begin{figure}[H]
\centering
\begin{tikzpicture}
\draw [thick] (0,0) node [anchor=east]{$1$} -- (8,0) node [anchor= west]{$p$};
\draw [thick] (2,2) -- (8,0) node [anchor=west]{$p$};
\draw [thick] (6,2) .. controls (6.05,1.8) and (6.3,1.5) .. (6.5,1.2);
\draw [thick] (0,0) .. controls (1.4,0.7) and (1.8,1.4) .. (2,2) node [anchor=south]{$p_1=xc_1$};
\draw [thick] (2,2) -- (6,2) node [anchor=south]{$p_1 c_2 = x c_1 c_2$};
\draw [thick] (2,2) -- (2,0) node [anchor=north]{$x$};
\draw [thick] (6,2) --(6,0) node [anchor=north]{$y$};
\draw [thick] (6.5,1.2) .. controls (6.8,0.8) and (7.17,0.41) .. (8,0);
\draw [thick,dashed] (3.5,2) .. controls (3.46,1.75) .. (3.5,1.5);
\draw [thick,dashed] (4.5,2) .. controls (4.44,1.59) .. (4.5,1.18);
\draw [thick,dashed] (3.5,1.5) .. controls (4,0.9) and (4.7,0.5) .. (6,0);
\draw [thick,dashed] (4.5,1.18) .. controls (4.9,0.85) and (5.5,0.5) .. (6.8,0);
\filldraw[black] (0,0) circle (0.8 pt);
\filldraw[black] (8,0) circle (0.8 pt);
\filldraw[black] (2,2) circle (0.8 pt);
\filldraw[black] (6,2) circle (0.8 pt);
\filldraw[black] (2,0) circle (0.8 pt);
\filldraw[black] (6,0) circle (0.8 pt);
\filldraw[black] (3.5,2) node [anchor=south]{$y''$} circle (0.8 pt);
\filldraw[black] (4.5,2) node [anchor=south]{$z$} circle (0.8 pt);
\filldraw[black] (3.5,1.5) node [anchor=north east]{$y'$} circle (0.8 pt);
\filldraw[black] (4.5,1.18) node [anchor=west]{$z_1$} circle (0.8 pt);
\filldraw[black] (6.8,0) node [anchor=north]{$w_1$} circle (0.8 pt);
\end{tikzpicture}
\caption{(i) $z_1 \in [p_1, p]$.}
\label{le4-pic8}
\end{figure}

\smallskip

{\bf (ii)} Let $z_1 \in [p_1 c_2, p]$. It follows that $d(z, p_1 c_2) \leqslant (p_1 \cdot p)_{p_1 c_2}$.

If $d(z, p_1 c_2) \leqslant (p_1 \cdot 1)_{p_1 c_2}$, then we can consider the triangle $\Delta(1, p_1 c_2, p_1)$, which is $C_1$-thin, and find a point on $[1, p_1 c_2]$ that is $C_1$-close to $z$. Repeating the argument from {\bf (i)} with the triangles $\Delta(1, p_1 c_2, p_1)$, $\Delta(1, y, p_1 c_2)$, and $\Delta(1, p_1, x)$, we obtain a bound on the distance from $z$ to $\gamma(x, y)$.

Suppose that $d(z, p_1 c_2) > (p_1 \cdot 1)_{p_1 c_2}$. Hence, $d(z, p_1) \leqslant (1 \cdot p_1 c_2)_{p_1}$ and there exists $r_1 \in [1, p_1]$ such that $d(z, r_1) \leqslant C_1$. The triangle $\Delta(1, p_1, x)$ is $C_1$-thin, hence, there exists $r_2 \in \gamma(1, x)$ such that $d(r, r_2) \leqslant C_1 + C_3$. The triangle $\Delta(y, p, p_1 c_2)$ is also $C_1$-thin, so there exists $z_2 \in \gamma(y, p)$ such that $d(z_1, z_2) \leqslant C_1 + C_3$. But then we obtain
$$d(x, y) \leqslant d(r_2, z_2) \leqslant 2(C_1 + C_3) + 2C_1,$$
which implies that
$$d(p_1, p_1 c_2) \leqslant |c_1| + d(x, y) + |c_3| \leqslant 2(C_1 + C_3) + 2C_1 + 2C_3 = 4(C_1 + C_3)$$
and the distance from $z$ to $\gamma(x, y)$ is bounded.

\begin{figure}[H]
\centering
\begin{tikzpicture}
\draw [thick] (0,0) node [anchor=east]{$1$} -- (8,0) node [anchor= west]{$p$};
\draw [thick] (2,2) -- (8,0) node [anchor=west]{$p$};
\draw [thick] (0,0) .. controls (1.4,0.7) and (1.8,1.4) .. (2,2) node [anchor=south]{$p_1 = x c_1$};
\draw [thick] (6,2) .. controls (6.05,1.8) and (6.3,1.5) .. (6.5,1.2);
\draw [thick] (6.5,1.2) .. controls (6.8,0.8) and (7.17,0.41) .. (8,0);
\draw [thick] (2,2) -- (6,2) node [anchor=south]{$p_1 c_2$};
\draw [thick] (2,2) -- (2,0) node [anchor=north]{$x$};
\draw [thick] (6,2) --(6,0) node [anchor=north]{$y$};
\draw [thick,dashed] (1.41,1) .. controls (2.5,1.3) and (3.5,1.7) .. (4,2);
\draw [thick,dashed] (6.5,1.2) .. controls (5.8,1.3) and (4.7,1.7) .. (4,2);
\draw [thick,dashed] (6.5,1.2) .. controls (6.41,0.6) .. (6.5,0);
\draw [thick,dashed] (1.41,1) .. controls (1.47,0.5) .. (1.41,0);
\filldraw[black] (0,0) circle (0.8 pt);
\filldraw[black] (8,0) circle (0.8 pt);
\filldraw[black] (2,2) circle (0.8 pt);
\filldraw[black] (6,2) circle (0.8 pt);
\filldraw[black] (2,0) circle (0.8 pt);
\filldraw[black] (6,0) circle (0.8 pt);
\filldraw[black] (6.5,1.2) circle (0.8 pt)node [anchor= west]{$z_1$};
\filldraw[black] (4,2) circle (0.8 pt) node [anchor=south]{$z$};
\filldraw[black] (1.41,1) circle (0.8 pt) node [anchor= east]{$r_1$};
\filldraw[black] (1.41,0) circle (0.8 pt) node [anchor= north]{$r_2$};
\filldraw[black] (6.5,0) circle (0.8 pt) node [anchor= north]{$z_2$};
\end{tikzpicture}
\caption{(ii) $z_1 \in [p_1 c_2, p]$.}
\label{le4-pic9}
\end{figure}
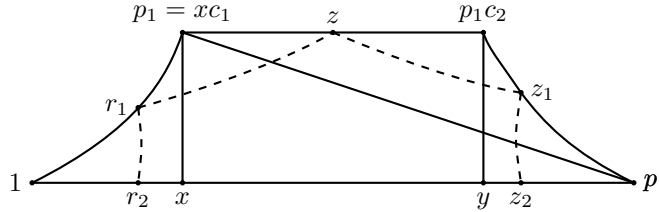

\smallskip

{\bf Case III.} Let $k = 4$. That is, $x^{-1} y = c_1 c_2 c_3 c_4$.

Similarly to {\bf Case II}, we obtain $p = (p_1 p_2) (c_1 c_2 c_3 c_4) (q_2^{-1} q_1^{-1})$, which is possible only if $p_2 c_1 \in B_P$ and $c_3 q_2^{-1} \in B_P$, and without loss of generality we assume that $p_2 c_1 = c_4 q_2^{-1} = 1$. Hence, $p = p_1 (c_2 c_3) q_1^{-1}$.

If $p_1 c_2 \notin P$ and $c_3 q_1^{-1} \notin P$, then the product $p_1 c_2 c_3 q_1^{-1}$ is reduced and it cannot be equal to a single element of $P$, which is a contradiction. Suppose $p_1 c_2 \in P$. 

If $c_3 q_1^{-1} \notin P$, then $(p_1 c_2) c_3$ must be in $P$ because otherwise $(p_1 c_2) c_3 q_1^{-1}$ is reduced and we obtain a contradiction. So from the axiom (P6), it follows that $p_1 c_2 \in B_P$. Next, the product $(p_1 c_2 c_3) q_1^{-1}$ cannot be reduced because it is equal to a single element of $P$, hence, $(p_1 c_2 c_3) q_1^{-1} \in P$ and again from (P6) it follows that $(p_1 c_2) c_3 \in B_P$. But we already have that $p_1 c_2 \in B_P$, so it implies that $c_3 \in B_P$ since $B_P$ s a subgroup of $P$ (and $G$). This is a contradiction with the fact that the product $c_1 c_2 c_3 c_4$ is reduced.

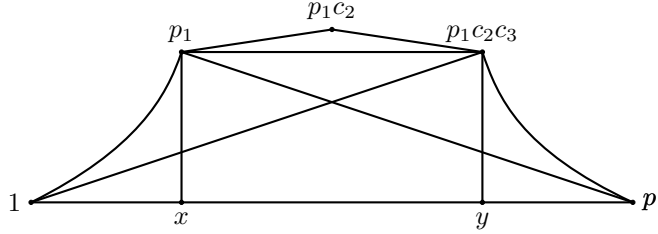
\begin{figure}[H]
\centering
\begin{tikzpicture}
\draw [thick] (0,0) node [anchor=east]{$1$} -- (8,0) node [anchor= west]{$p$};
\draw [thick] (2,2) -- (8,0) node [anchor=west]{$p$};
\draw [thick] (0,0) -- (6,2);
\draw [thick] (2,2) -- (6,2);
\draw [thick] (0,0) .. controls (1.4,0.7) and (1.8,1.4) .. (2,2) node [anchor=south]{$p_1$};
\draw [thick] (2,2) -- (4,2.3) node [anchor=south]{$p_1 c_2$};
\draw [thick] (4,2.3) -- (6,2) node [anchor=south]{$p_1 c_2 c_3$};
\draw [thick] (2,2) -- (2,0) node [anchor=north]{$x$};
\draw [thick] (6,2) -- (6,0) node [anchor=north]{$y$};
%\draw [thick,dashed] (4.4,1.2) node [anchor=south]{$s$} .. controls (4.5,0.6) .. (4.4,0) node [anchor= north]{$t$};
\draw [thick](6,2) .. controls (6.2,1.4) and (6.5,0.7) .. (8,0);
%\draw [thick] (6.5,1.2) .. controls (6.8,0.8) and (7.17,0.41) .. (8,0);
%\draw [thick,dashed] (4.4,1.2) node [anchor=south]{$s$} .. controls (5.45,1.1) .. (6.5,1.2);
%\draw [thick,dashed] (6.5,1.2) .. controls (6.42,0.6) .. (6.5,0) node [anchor=north]{$t_1$};
\filldraw [black] (4,2.3) circle (0.8 pt);
\filldraw [black] (0,0) circle (0.8 pt);
\filldraw [black] (8,0) circle (0.8 pt);
\filldraw [black] (2,2) circle (0.8 pt);
\filldraw [black] (6,2) circle (0.8 pt);
\filldraw [black] (2,0) circle (0.8 pt);
\filldraw [black] (6,0)  circle (0.8 pt);
%\filldraw [black] (6.5,1.2) circle (0.8 pt);
%\filldraw [black] (6.5,0) circle (0.8 pt);
%\filldraw [black] (4.4,0) circle (0.8 pt);
\end{tikzpicture}
\caption{Case III. $x^{-1} y = c_1 c_2 c_3 c_4$.}
\label{le4-pic10}
\end{figure}

Thus, we have $p_1 c_2 \in P$ and $c_3 q_1^{-1} \in P$. Note that if $(p_1 c_2) c_3 \notin P$, then from $(p_1 c_2) c_3 q_1^{-1} \in P$ and the axiom (P6) it follows that $c_3 q_1^{-1} \in B_P$. If $(p_1 c_2) c_3 \in P$, then again from (P6) it follows that $p_1 c_2 \in B_P$. That is, $p = (p_1 c_2) (c_3 q_1^{-1})$ and either $p_1 c_2 \in B_P$, or $c_3 q_1^{-1} \in B_P$. Without loss of generality assume that $p_1 c_2 \in B_P$.

\smallskip

{\bf (a)} Let $t \in [p_1, p_1 c_2] \cup [p_1 c_2, p_1 c_2 c_3]$. 

Consider the geodesic triangle $\Delta(p_1, p_1 c_2, p_1 c_2 c_3)$, which is $C_1$-thin by the axiom (H1). Since $c_2 c_3$ is a reduced $P$-product, it follows that $(p_1 \cdot p_1 c_2 c_3)_{p_1 c_2} \leqslant C_2$ by the axiom (H2) and there exists $s \in [p_1, p_1 c_2 c_3]$ such that $d(t, s) \leqslant C_1 + C_2$. Next, the triangle $\Delta(p_1, p, p_1 c_2 c_3)$ is also $C_1$-thin, so there exists $v_1 \in [p_1, p] \cup [p_1 c_2 c_3, p]$ such that $d(s, v_1) \leqslant C_1$. If $v_1 \in [p_1, p]$, then since the triangle $\Delta(x, p_1, p)$ is $C_1$-thin too, there exists $v_2 \in \gamma(x, p)$ such that $d(v_1, v_2) \leqslant C_1$. If $v_1 \in [p_1 c_2 c_3, p]$, then since the triangle $\Delta(y, p_1 c_2 c_3, p)$ is $C_1$-thin too and $d(y, p_1 c_2 c_3) \leqslant C_3$, there exists $v_2 \in \gamma(y, p)$ such that $d(v_1, v_2) \leqslant C_1 + C_3$. In any case, $v_2 \in \gamma(x, p)$ and
$$d(t, v_2) \leqslant d(t, s) + d(s, v_1) + d(v_1, v_2) \leqslant (C_1 + C_2) + C_1 + (C_1 + C_3)$$
$$= 3C_1 + C_2 + C_3.$$
If $v_2 \in \gamma(x, y)$, then we are done. Assume that $v_2 \in \gamma(y, p)$.

Using a similar argument, symmetrically considering the triangles $\Delta(1, p_1, p_1 c_2 c_3)$, $\Delta(1, y, p_1 c_2 c_3)$, and $\Delta(1, p_1, x)$, all of which are $C_1$-thin, we can find a point $w_2 \in \gamma(1, y)$ such that 
$$d(t, w_2) \leqslant d(t, s) + d(s, w_2) \leqslant 3C_1 + C_2 + C_3.$$
Again, if $w_2 \in \gamma(x, y)$, then we are done. Assuming that $w_2 \in \gamma(1, x)$ we obtain
$$d(x, y) \leqslant d(w_2, v_2) \leqslant d(t, w_2) + d(t, v_2)$$
$$\leqslant (3C_1 + C_2 + C_3) + (3C_1 + C_2 + C_3) = 6C_1 + 2C_2 + 2C_3,$$
which implies that
$$d(p_1, p_1 c_2 c_3) \leqslant d(p_1, x) + d(x, y) + d(y + p_1 c_2 c_3)$$
$$\leqslant C_3 + (6C_1 + 2C_2 + 2C_3) + C_3 = 6C_1 + 2C_2 + 4C_3.$$
Finally, since $c_2 c_3$ is a reduced $P$-product, from the axiom (H2) we obtain
$$2C_2 \geqslant 2(p_1 \cdot p_1 c_2 c_3)_{p_1 c_2} = d(p_1, p_1 c_2) + d(p_1 c_2, p_1 c_2 c_3) - d(p_1, p_1 c_2 c_3)$$
$$\geqslant d(p_1, p_1 c_2) + d(p_1 c_2, p_1 c_2 c_3) - (6C_1 + 2C_2 + 4C_3),$$
hence,
$$|c_2| + |c_3| = d(p_1, p_1 c_2) + d(p_1 c_2, p_1 c_2 c_3) \leqslant 6C_1 + 4C_2 + 4C_3,$$
that is, both $c_2$ and $c_3$ have bounded lengths, and
$$d(t, x) \leqslant d(t, p_1) + d(p_1, x) \leqslant |c_2| + |c_3| + d(p_1, x)$$
$$\leqslant (6C_1 + 4C_2 + 4C_3) + C_3 = 6C_1 + 4C_2 + 5C_3.$$

\smallskip

{\bf (b)} Let $z \in \gamma(x, y)$.

The triangle $\Delta(1, y, p_1 c_2 c_3)$ is $C_1$-thin and $d(y, p_1 c_2 c_3) \leqslant C_3$, hence, there exists $w \in [1, p_1 c_2 c_3]$ such that $d(z, w) \leqslant C_1$ (in case when $w \in [y, p_1 c_2 c_3]$ we immediately obtain $d(z, p_1 c_2 c_3) \leqslant C_1 + C_3$ and we are done). Next, the triangle $\Delta(1, p_1, p_1 c_2 c_3)$ is also $C_1$-thin, so there exists $r \in [1, p_1] \cup [p_1, p_1 c_2 c_3]$ such that $d(w, r) \leqslant C_1$. If $r \in [p_1, p_1 c_2 c_3]$, then $r$ is $(C_1 + C_2)$-close to $\Delta(p_1, p_1 c_2, p_1 c_2 c_3)$, and we can find a point $u \in [p_1, p_1 c_2] \cup [p_1 c_2, p_1 c_2 c_3]$ such that $d(r, u) \leqslant C_1 + C_3$, which implies that
$$d(z, u) \leqslant d(z, w) + d(w, r) + d(r, u) \leqslant C_1 + C_1 + (C_1 + C_2) = 3C_1 + C_2.$$
If $r \in [1, p_1]$, then, since the triangle $\Delta(1, p_1, x)$ is $C_1$-thin and $d(x, p_1) \leqslant C_3$, it follows that there exists $r_1 \in \gamma(1, x)$ such that $d(r, r_1) \leqslant C_1 + C_3$. Hence, we have
$$d(x, z) \leqslant d(r_1, z) \leqslant d(r_1, r) + d(r, w) + d(w, z) \leqslant (C_1 + C_3) + C_1 + C_1 = 3C_1 + C_3.$$

\smallskip

{\bf Case IV.} If $k > 4$, then we have $p = (p_1 p_2) (c_1 \cdots c_k) (q_2^{-1} q_1^{-1})$ and assuming $p_2 c_1 = c_k q_2^{-1} = 1$ we obtain
$$p = p_1 (c_2 \cdots c_{k-1}) q_1^{-1},$$
$c_2 \cdots c_{k-1}$ is a reduced $P$-product of length at least three. In this case the product $p_1 (c_2 \cdots c_{k-1}) q_1^{-1}$ cannot be reduced to a single element of $P$ and we obtain a contradiction.

\end{proof}

\subsection{Quasi-geodesics defined by $P$-reduced products}
\label{subsec:hyp-crit}

Our proof of Theorem \ref{th:main} is going to rely on the following criterion for hyperbolicity (see \cite{Bowditch:2006, Hamenstadt:2007}).

\begin{theorem}\cite[Proposition 3.5]{Hamenstadt:2007}
\label{th:criterion}
Let $(X, d)$ be a geodesic metric space. Assume that there is a number $D > 0$ and for every pair of points $x, y \in X$ there is an path $\eta(x, y) : [0, 1] \to X$ connecting $\eta(x, y)(0) = x$ to $\eta(x, y)(1) = y$ so that the following conditions are satisfied.
\begin{enumerate}
\item[(a)] If $d(x, y) \leqslant 1$, then the diameter of $\eta(x, y)$ is at most $D$. 
\item[(b)] For $x, y \in X$ and $0 \leqslant s \leqslant t \leqslant 1$, the Hausdorff distance between $\eta(x, y)[s, t]$ and $\eta(\eta(x, y)(s), \eta(x, y)(t))[0, 1]$ is at most $D$.
\item[(c)] For any $x, y, z \in X$, the set $\eta(x, y)$ is contained in the $D$-neighborhood of $\eta(x, z)[0, 1] \cup \eta(z, y)[0, 1]$.
\end{enumerate}
Then $(X, d)$ is $\delta$-hyperbolic for a number $\delta > 0$ only depending on $D$.
\end{theorem}

Note that $\Gamma(G, S)$ is geodesic metric space. If $x, y \in \Gamma(G, S)$, then $g = x^{-2} y \in G$ and there exists a reduced $P$-product $p = p_1 p_2 \cdots p_n$ representing $g$. This product defines the path 
$$\mathcal{P}_g = [x, x p(1)] \cup [x p(1), x p(2)] \cup \cdots \cup [x p(n-1), x p(n)]$$
from $x$ to $y$. Such a path is not unique since there can be many reduced $P$-products representing $g$, as well as many geodesics connecting adjacent points in the sequence $1, p(1), p(2), \ldots, p(n)$. Following the notation introduced in Theorem \ref{th:criterion}, for every $x, y \in \Gamma(G, S)$ let $\eta(x, y)$ denote the collection of all paths $\mathcal{P}_g$ defined above for $g = x^{-1} y$. If $x = y$, then $\eta(x, y)$ contains only a trivial path equal to a single point.

Now we are ready to prove the main result of the paper.

\begin{proof}[Proof of Theorem \ref{th:main}]
In order to prove hyperbolicity of $\Gamma(G, S)$ we are going to verify all the conditions listed in the Theorem \ref{th:criterion}.

\smallskip

{\bf (a)} Let $x, y \in \Gamma(G, S)$ be such that $d(x, y) \leqslant 1$. If $d(x, y) = 0$, then $\eta(x, y)$ contains only a trivial path and its diameter is $0$. If $d(x, y) = 1$, then $x^{-1} y = g \in S \cup S^{-1}$ and from Lemma \ref{le:fin-diameter} it follows that there exists a constant $D_1 > 0$ such that the diameter of any $p_g \in \eta(x, y)$ is bounded by $D_1$ and the condition (a) of Theorem \ref{th:criterion} holds for $D = D_1$.

\smallskip

{\bf (b)} Let $x, y \in \Gamma(G, S)$, let $g = x^{-1} y$, and $\mathcal{P}_g \in \eta(x, y)$. The path $\mathcal{P}_g$ corresponds to some reduced $P$-product $p_1 p_2 \ldots p_n$.

Suppose $s$ and $t$ are points on $\mathcal{P}_g$. We can assume that $s \in \gamma_1 = [p(i-1), p(i)]$ and $t \in \gamma_2 = [p(j-1), p(j)]$ (without loss of generality we can assume that $i \leqslant j$).

Note that the element $r = s^{-1} t \in G$ can also be represented by a reduced $P$-product and we take the corresponding $\mathcal{P}_r \in \eta(s, t)$. 

At the same time, from the axiom (H3) it follows that the element 
$f = s^{-1} p(i)$ can be represented by a reduced $P$-product $f_1 f_2$, where $|f_1| \leqslant C_3$ and let $\mathcal{P}_f \in \eta(s, p(i))$. Similarly, $h = p(j-1)^{-1} t$ can be represented by a reduced $P$-product $h_1 h_2$, where $|h_2| \leqslant C_3$ and we take a path $\mathcal{P}_h \in \eta(p(j-1), t)$. 

Finally, the element $g_0 = p(i)^{-1} p(j-1)$ is represented by the reduced $P$-product $p_{i+1} \cdots p_{j-1}$. The path
$$\mathcal{P}_{g_0} = [p(i), p(i+1)] \cup [p(i+1), p(i+2)] \cup \cdots \cup [p(j-2), p(j-1)]$$
is a subpath of $\mathcal{P}_h$ and $\mathcal{P}_{g_0} \in \eta(p(i), p(j-1))$. Note that both $\mathcal{P}_r$ and the concatenation of paths $\mathcal{P}_f \mathcal{P}_{g_0} \mathcal{P}_h$ connect $s$ and $t$ in $\Gamma(G, S)$.

\begin{figure}[H]
\centering
\begin{tikzpicture}
\draw [thick,->] (0,0.5) node [anchor=north ]{$x$} -- (1,1) node [anchor=south ]{$p(1)$};
\draw [thick,->,dashed] (1,1) -- (3,1.5) node [anchor=north ]{$p(i-1)$};
\draw [thick,->] (3,1.5) -- (5,2) node [anchor=south ]{$p(i)$};
\draw [thick,->] (5,2) -- (5.7,2);
\draw [thick,->,dashed] (5.7,2) -- (6.5,2) node [anchor=north east ]{$\mathcal{P}_{g_0}$};
\draw [thick,->] (6.5,2) -- (7.5,2) node [anchor=south ]{$p(j-1)$};
\draw [thick,->] (7.5,2) -- (9.5,1.5) node [anchor=north ]{$p(j)$};
\draw [thick,->,dashed] (9.5,1.5) -- (11.5,1) node [anchor=south ]{$p(n-1)$};
\draw [thick,->] (11.5,1) -- (12.5,0.5) node [anchor=north ]{$y$};
\filldraw[black] (3.5,1.625) circle (1.5pt) node [anchor=south ]{$s$};
\filldraw[black] (9,1.625) circle (1.5pt) node [anchor=south ]{$t$};
\filldraw[black] (0,0.5) circle (1.5pt);
\filldraw[black] (12.5,0.5) circle (1.5pt);
\filldraw[black] (6,0.75) circle (0 pt) node [anchor=north ]{$\mathcal{P}_r$};
\draw [thick,->] (3.5,1.625) -- (4,1);
\draw [thick,->] (4,1) -- (5,0.8);
\draw [thick,->,dashed] (5,0.8) -- (7,0.7);
\draw [thick,->] (7,0.7) -- (8,0.8);
\draw [thick,->] (8,0.8) -- (9,1.625);
\draw [thick,->] (3.5,1.625) -- (4.2,1.5) node [anchor=north west ]{$\mathcal{P}_f$};
\draw [thick,->] (4.2,1.5) -- (5,2);
\draw [thick,->] (7.5,2) -- (8.2,1.5) node [anchor=north east ]{$\mathcal{P}_h$};
\draw [thick,->] (8.2,1.5) -- (9,1.625);
\end{tikzpicture}
\caption{Condition (b).}
\end{figure}
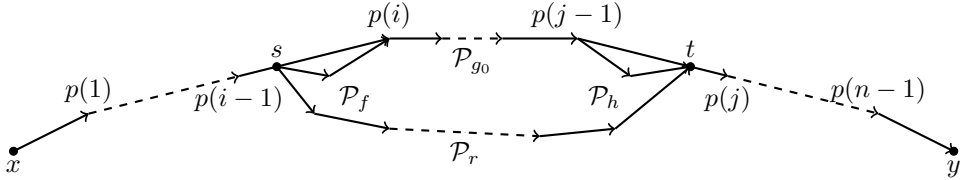

Our goal is to find a bound on the Hausdorff distance between $\gamma_1(s, p(i)) \cup \mathcal{P}_{g_0} \cup \gamma_2(p(j-1),t)$
and the path $\mathcal{P}_r$.

Note that the product $(f_1 f_2) (p_{i+1} \cdots p_{j-1})$ is not reduced, but there exists a reduced product $w = w_1 \cdots w_k$ representing $w = s^{-1} p(j-1)$ such that
$$(f_1 f_2) (p_{i+1} \cdots p_{j-1}) = w_1 \cdots w_k.$$
Let $\mathcal{P}_w \in \eta(s, p(j-1))$.

Now, the Hausdorff distance between $\gamma_1(s, p(i))$ and $\mathcal{P}_f$ is bounded by $C_1 + C_3$, which follows from the axioms (H1) and (H3). From Lemma \ref{le:thin-quads} it follows that the Hausdorff distance between $\mathcal{P}_f \cup \mathcal{P}_{g_0}$ and $\mathcal{P}_w$ is bounded by $D_4$. It follows that the Hausdorff distance between $\gamma_1(s, p(i)) \cup \mathcal{P}_{g_0}$ and  $\mathcal{P}_w$ is bounded by $D_4 + C_1 + C_3$. Similarly, the Hausdorff distance between $\gamma_2(p(j-1),t)$ and $\mathcal{P}_h$ is bounded by $C_1 + C_3$
and the Hausdorff distance between $\mathcal{P}_w \cup \mathcal{P}_h$ and $\mathcal{P}_r$ is bounded by $D_4$, which follows from Lemma \ref{le:thin-quads}. Hence, the Hausdorff distance between $\mathcal{P}_w \cup \gamma_2(p(j-1),t)$ and $\mathcal{P}_r$ is bounded by $D_4 + C_1 + C_3$. Finally, we obtain that the Hausdorff distance between $\gamma_1(s, p(i)) \cup \mathcal{P}_{g_0} \cup \gamma_2(p(j-1),t)$ and $\mathcal{P}_r$ is bounded by $2(D_4 + C_1 + C_3)$ and the condition (b) of Theorem \ref{th:criterion} holds for $D = 2(D_4 + C_1 + C_3)$.

\smallskip

{\bf (c)} Let $x, y, z \in \Gamma(G, S)$ and let $f = x^{-1} y,\ g = y^{-1} z,\ h = x^{-1} z$. If $\mathcal{P}_f \in \eta(x, y),\ \mathcal{P}_g \in \eta(y, z),\ \mathcal{P}_h \in \eta(x, z)$, then from Lemma \ref{le:thin-tiangles} it follows that there exists $D_3 > 0$ such that the triangle in $\Gamma(G, S)$ formed by the paths $\mathcal{P}_f, \mathcal{P}_g, \mathcal{P}_h$ is $D_3$-thin. In particular, it implies that $\mathcal{P}_f$ is contained in the
$D_3$-neighborhood of $\mathcal{P}_g \cup \mathcal{P}_h$. Hence,  the condition (c) of Theorem \ref{th:criterion} holds for $D = D_3$.

\smallskip

Since all the conditions of Theorem \ref{th:criterion} are satisfied for $D = \max\{D_1, 2(D_4 + C_1 + C_3), D_3\}$, all geodesic triangles in $\Gamma(G, S)$ are $\delta$-thin for some constant $\delta > 0$ and $\Gamma(G, S)$ is a $\delta$-hyperbolic space. Hence, $G \simeq U(P)$ is a word hyperbolic group.
\end{proof}

\end{document}